\documentclass[12pt]{amsart}
\usepackage{changepage}
\usepackage{multicol}
\usepackage{SSdefn}
\usepackage{SSams}
\usepackage{eucal}

\makeatletter
\def\l@part{\@tocline{0}{4pt plus2pt}{0pt}{}{\bfseries}}
\makeatother

\newcommand{\init}{\mathrm{init}}

\newcommand{\Conf}{\mathrm{Conf}}
\newcommand{\ab}{\mathrm{ab}}

\newcommand{\SYM}{\widehat{\rm Sym}}
\DeclareMathOperator{\chr}{char}

\newcommand{\OI}{\mathbf{OI}}

\newcommand{\FA}{\mathbf{FA}}
\newcommand{\FI}{\mathbf{FI}}
\newcommand{\FS}{\mathbf{FS}}
\newcommand{\FB}{\mathbf{FB}}
\newcommand{\FWS}{\mathbf{FWS}}

\newcommand{\OWS}{\mathbf{OWS}}

\newcommand{\Set}{\mathbf{Set}}

\title{Representations of categories of $G$-maps}
\date{September 6, 2016}
\subjclass[2010]{%
05A15, 
13P10, 
16P40, 
18A25, 
68Q70.
}

\author{Steven V Sam}
\address{Department of Mathematics, University of California, Berkeley, CA}
\curraddr{Department of Mathematics, University of Wisconsin, Madison, WI}
\email{\href{mailto:svs@math.wisc.edu}{svs@math.wisc.edu}}
\urladdr{\url{http://math.wisc.edu/~svs/}}

\author{Andrew Snowden}
\address{Department of Mathematics, University of Michigan, Ann Arbor, MI}
\email{\href{mailto:asnowden@umich.edu}{asnowden@umich.edu}}
\urladdr{\url{http://www-personal.umich.edu/~asnowden/}}

\thanks{S. Sam was supported by a Miller research fellowship and NSF grant DMS-1500069.}
\thanks{A. Snowden was supported by NSF grants DMS-1303082 and DMS-1453893.}

\begin{document}

\begin{abstract}
We study representations of wreath product analogues of categories of finite sets. This includes the category of finite sets and injections (studied by Church, Ellenberg, and Farb) and the opposite of the category of finite sets and surjections (studied by the authors in previous work). We prove noetherian properties for the injective version when the group in question is polycyclic-by-finite and use it to deduce general twisted homological stability results for such wreath products and indicate some applications to representation stability. We introduce a new class of formal languages (quasi-ordered languages) and use them to deduce strong rationality properties of Hilbert series of representations for the surjective version when the group is finite. 
\end{abstract}

\maketitle

\tableofcontents

\section{Introduction}

In \cite{fimodules}, Church, Ellenberg, and Farb studied representations of the category $\FI$, consisting of finite sets with injective maps, and found numerous applications to topology and algebra. In \cite{catgb}, we studied representations of the closely related category $\FS^{\op}$ (in addition to many other examples), where $\FS$ is the category of finite sets with surjective maps, and also found several applications (e.g., to the Lannes--Schwartz artinian conjecture and to $\Delta$-modules). In this paper, we study generalizations of $\FI$ and $\FS^{\op}$ in which a group $G$ has been added to the mix. Our main tool is the theory developed in \cite{catgb}; in fact, the example $\FS_G^{\op}$ was a primary source of motivation for that paper. In the rest of the introduction, we summarize our motivations and results.

\subsection{$G$-maps}

Let $G$ be a group. A {\bf $G$-map} between finite sets $R$ and $S$ is a pair $(f, \rho)$ consisting of functions $f \colon R \to S$ and $\rho \colon R \to G$. If $(f, \rho) \colon R \to S$ and $(g, \sigma) \colon S \to T$ are two $G$-maps, their composition $(h, \tau) \colon R \to T$ is defined by $h=g \circ f$ and $\tau(x) = \sigma(f(x))\rho(x)$ (the product taken in $G$). We let $\FA_G$ be the category whose objects are finite sets and whose morphisms are $G$-maps. We let $\FI_G$ (resp.\ $\FS_G$) be the subcategory where the function $f$ is injective (resp.\ surjective). We note that the automorphism group of the set $[n]=\{1, \ldots, n\}$ in any of these categories is the wreath product $S_n \wr G$. Thus a representation of any of these categories can be thought of as a sequence $(M_n)_{n \ge 0}$, where $M_n$ is a representation of $S_n \wr G$, equipped with certain transition maps between $M_n$ and $M_{n+1}$. (The kind of transition maps depends on the category.) These representations are the subject of this paper.

\begin{remark}
A $G$-map $R \to S$ is the same as a $G$-equivariant map $R \times G \to S \times G$. Thus $\FA_G$ is equivalent to the category whose objects are free $G$-sets with finitely many orbits, and whose morphisms are $G$-equivariant functions.
\end{remark}

\subsection{The category $\FI_G$}

Usually, noetherianity is the first property one wants to establish about the representation theory of a category. (See \S \ref{ss:rep} for the definition of ``noetherian'' in this context.) For $\FI$, this was proved in \cite{delta-mod,fimodules,fi-noeth,catgb}, in varying levels of generality. Our main result about $\FI_G$ characterizes when representations are noetherian:

\begin{theorem}
\label{mainthm:fi1}
Let $\bk$ be a ring. Then $\Rep_{\bk}(\FI_G)$ is noetherian if and only if the group algebra $\bk[G^n]$ is left-noetherian for all $n \ge 0$.
\end{theorem}

Recall that a group $G$ is {\bf polycyclic} if it has a finite composition series 
\[
1 = G_0 \subseteq G_1 \subseteq \cdots \subseteq G_r = G
\]
such that $G_i / G_{i-1}$ is cyclic for $i=1,\dots,r$, and it is {\bf polycyclic-by-finite} (or {\bf virtually polycyclic}) if it contains a polycyclic subgroup of finite index. It is known \cite[\S 2.2, Lemma 3]{hall} that the group ring of a polycyclic-by-finite group over a left-noetherian ring is left-noetherian (there it is stated for the integral group ring, but the proof works for any left-noetherian coefficient ring). In fact, there are no other known examples of noetherian group algebras, but see \cite{ivanov} for related results. Since a finite product of polycyclic-by-finite groups is again polycyclic-by-finite, the above theorem gives:

\begin{corollary}
\label{cor:polycyclic}
Let $G$ be a polycyclic-by-finite group and let $\bk$ be a left-noetherian  ring. Then $\Rep_{\bk}(\FI_G)$ is noetherian.
\end{corollary}

When $G$ is a finite group, we prove a stronger result:

\begin{theorem}
\label{mainthm:fi2}
If $G$ is finite then the category $\FI_G$ is quasi-Gr\"obner.
\end{theorem}

``Quasi-Gr\"obner'' is a purely combinatorial condition on a category, introduced in \cite{catgb} (and recalled in \S \ref{ss:rep} below), that implies noetherianity of the representation category. Thus Theorem~\ref{mainthm:fi2} implies Theorem~\ref{mainthm:fi1} when $G$ is finite, as stated. However, quasi-Gr\"obner gives more than just noetherianity: it implies that representations admit a theory of Gr\"obner bases, in an appropriate sense, and thus computations with representations can be carried out algorithmically (at least in principle).

The proof of Theorem~\ref{mainthm:fi2} is an easy consequence of the theory developed in \cite{catgb}. The proof of Theorem~\ref{mainthm:fi1} is more involved. The key input is the fact that $\Rep_{\bk'}(\FI)$ is noetherian whenever $\bk'$ is left-noetherian. This is applied with $\bk'=\bk[G^n]$, so even if one only cares about Theorem~\ref{mainthm:fi1} when $\bk$ is a field, the proof uses $\FI$-modules over non-commutative rings. In fact, this is the first real application of $\FI$-modules over non-commutative rings that we know of.

There is one additional result on $\FI_G$-modules worth mentioning:

\begin{theorem}
\label{mainthm:fi3}
Suppose that $G$ is finite and $\bk$ is a splitting field for $G$ in which the order of $G$ is invertible. Then representations of $\FI_G$ are equivalent to representations of $\FI \times \FB^r$, where $r$ is the number of non-trivial irreducible representations of $G$ over $\bk$.
\end{theorem}

Here $\FB$ is the groupoid of finite sets, i.e., the category of finite sets with bijections as morphisms. Thus, in the context of the theorem, the theory of $\FI_G$-modules reduces to the theory of $\FI$-modules.

\subsection{The category $\FS_G$}

We only study $\FS_G$ when $G$ is a finite group. The first result is about noetherianity, and is an easy consequence of the theory of \cite{catgb}:

\begin{theorem}
\label{mainthm:fs1}
If $G$ is finite then $\FS_G^{\op}$ is quasi-Gr\"obner. In particular, if $\bk$ is left-noetherian, then $\Rep_{\bk}(\FS_G^{\op})$ is noetherian.
\end{theorem}

We next turn to Hilbert series. Suppose that $M$ is an $\FS_G^{\op}$-module over a field $\bk$, and suppose that $G$ has $r$ irreducible representations over $\bk$. We define a formal power series $\rH_M(\bt) \in \bQ \lbb t_1, \ldots, t_r \rbb$, called the {\bf Hilbert series} of $M$, that records the class of $M([n])$ in the representation ring of $G^n$ for all $n \ge 0$. Our main result is a rationality result for this series. The strongest and most general result takes some preparation to state, so we confine ourselves to the following simplified form here (which is a special case of Theorem~\ref{FSGhilb}):

\begin{theorem}
\label{mainthm:fs2}
Let $M$ be a finitely generated representation of $\FS_G^{\op}$ over an algebraically closed field $\bk$. Let $N$ be the exponent of the group $G$. Then $\rH_M(\bt)$ can be written in the form $f(\bt)/g(\bt)$, where $f$ and $g$ are polynomials in the $t_i$ with coefficients in $\bQ(\zeta_N)$, and $g$ factors over $\ol{\bQ}$ as $\prod_{k=1}^n (1-\lambda_k)$, where $\lambda_k$ is a $\bZ[\zeta_N]$-linear combination of the $t_i$. {\rm (}And $\zeta_N \in \ol{\bQ}$ is a primitive $N$th root of unity.{\rm )}
\end{theorem}

To paraphrase: if $M$ is a finitely generated $\FS_G^{\op}$-module then the representations $M([n])$ of $G^n$ satisfy recursive relations of a very particular form. We show this by connecting these Hilbert series for the generating functions of ``quasi-ordered languages'' which we introduce in this paper. To give a sense of how this connects to our previous paper \cite{catgb}: there we showed that finitely generated $\FS^\op$-representations have Hilbert series similar to those of $\FS_G^\op$, except with $\bZ$ in place of $\bZ[\zeta_N]$. The difference is that such Hilbert series are controlled by ``ordered languages'', and quasi-ordered languages are built from these with additional congruence conditions, which is why roots of unity appear.

This is by far the deepest result of the paper, and much of our work goes into its proof. The idea is to first use Brauer's theorem to reduce to the case where $G$ is a cyclic group whose order is invertible in $\bk$. Representations of $G$ are identified with vector spaces graded by the character group $\Lambda$ of $G$, and, in a similar fashion, representations of $\FS_G^{\op}$ with those of $\FWS_{\Lambda}^{\op}$, a certain category whose objects are finite sets in which each element is weighted by an element of $\Lambda$. An ordered version of this category is amenable to the theory of lingual structures developed in \cite{catgb}, which produces rationality results as above for Hilbert series.

\begin{remark}
We did not state any theorems about Hilbert series of $\FI_G$-modules. When $G$ is finite, one can use the fact that the Hilbert series of a finitely generated $\FI_G$-module is also the Hilbert series of a finitely generated $\FI$-module (use Proposition~\ref{prop:FIG-propertyF}), together with known results on Hilbert series of ordinary $\FI$-modules \cite[Theorem~B]{fi-noeth}, \cite[Corollary~7.1.7]{catgb}.
\end{remark}

\subsection{Applications and motivations}

\begin{itemize}
\item The category $\FI_{\bZ/2}$ is equivalent to the category $\FI_{\rm BC}$ defined in \cite[Defn.~1.2]{wilson}. It is possible to define and prove properties about modified versions of our categories to include her category $\FI_{\rm D}$; we leave the modifications to the reader. We point the reader to \cite{wilson-repstab, wilson} for applications of the category $\FI_{\rm BC}$. 

\item Applications of $\FI_G$-modules to topology and group theory are studied in \cite{casto}.

\item In \S\ref{sec:wreath-stab}, we show that wreath product generalizations of Murnaghan's stability theorem follow from the noetherian property for $\FI_G$-modules.

\item In \S\ref{ss:twisted-stab}, we use the machinery developed in \cite{putman-sam} applied to the category $\FI_G$ to prove general twisted homological stability results for wreath products $S_n \ltimes G^n$ when $G$ is a polycyclic-by-finite group.

\item Let $M$ be a simply-connected manifold of dimension at least $3$. In \cite{kupers-miller}, it is shown that the rational homotopy groups of the configuration spaces of $M$ satisfy representation stability, i.e., are finitely generated $\FI$-modules (see \cite{fimodules}). In \S\ref{ss:homotopy-conf}, we outline how this result might be extended when we drop the assumption that $M$ is simply-connected by using the category $\FI_{\pi_1(M)}$.

\item As we explain in \S \ref{ss:gendelta}, representations of the category $\FS_G^{\op}$ when $G$ is a symmetric group are essentially equivalent to $\Delta$-modules, in the sense of \cite{delta-mod}. This observation was our original source of motivation for studying the category $\FS_G^{\op}$: since $\FS_G^{\op}$-modules are much easier to think about than $\Delta$-modules, this point of view could be profitable. Indeed, our results on $\FS_G^{\op}$-modules imply the main theorems about $\Delta$-modules from \cite{delta-mod}, and more, and are more straightforward than the proofs given there. These theorems apply for any $G$, though, and so represent a significant generalization of the theory of $\Delta$-modules.

\item The main theorem on Hilbert series of $\Delta$-modules in \cite{delta-mod}, Theorem~B, was suspected to be suboptimal. Our original motivation in proving Theorem~\ref{mainthm:fs2} was to improve \cite[Theorem~B]{delta-mod}, which it does. We subsequently found an even stronger improvement, which appears in \cite[Theorem~9.2.3]{catgb}. However, \cite[Theorem~9.2.3]{catgb} is very specific to $\Delta$-modules, while Theorem~\ref{mainthm:fs2} applies to any group $G$.
\end{itemize}

\subsection{Open problems}

\subsubsection{Optimal results for Hilbert series of $\FS_G^{\op}$-modules} \label{optimalhilb}

Let $G$ be a finite group and let $\bk$ be an algebraically closed field. An interesting problem is to determine the smallest (or at least a small) field $F \subset \ol{\bQ}$ with the following property: if $M$ is a finitely generated $\FS_G^{\op}$-module over $\bk$ then $\rH_M(\bt)$ can be written in the form $f(\bt)/g(\bt)$ where $f \in F[\bt]$ and $g(\bt)$ factors as $\prod (1-\lambda_i)$ where each $\lambda_i$ is a linear combination of the $\bt$ with coefficients in the ring of integers $\cO_F$.

We prove that one can take $F \subset \bQ(\zeta_N)$, where $N$ is the exponent of $G$. In fact, we show that $F \subset \bQ(\zeta_N)$ if $G$ is ``$N$-good.'' For example, $G=S_n$ is 2-good (in good characteristic), and so $F=\bQ$ in this case. In characteristic~0, it follows from Example~\ref{hilbexample} that $F$ must contain the field generated by the character table of $G$. Can one always take $F$ to be the field generated by the Brauer character table of $G$?

\subsubsection{A reconstruction problem} \label{q:recon}
Theorem~\ref{mainthm:fi3} shows that $\Rep_{\bk}(\FI_G)$ knows very little of $G$ (only the number of irreducible representations). In contrast, $\Rep_{\bk}(\FS_G^{\op})$ knows a lot about $G$: for instance, it knows about tensor products of $G$ representations. It seems reasonable to think one could recover $G$ from $\Rep_{\bk}(G)$.

Here is a precise question. Let $G$ and $H$ be finite groups and let $\bk$ be an algebraically closed field. Suppose that $\Rep_{\bk}(\FS_G^{\op})$ and $\Rep_{\bk}(\FS_H^{\op})$ are equivalent as $\bk$-linear abelian categories. Are $G$ and $H$ isomorphic?

\subsection{Outline}

On a first reading, we encourage the reader to go through \S\ref{sec:fig} and \S\ref{sec:FSG} to see the main results. \S\ref{sec:fig} is mostly self-contained except for the background material on representations of categories in \S\ref{sec:background}. 

In \S \ref{sec:background}, we review material that we will use often, especially the main results from \cite{catgb}. In \S\ref{sec:fig} we investigate the category $\FI_G$, and prove Theorems~\ref{mainthm:fi1},~\ref{mainthm:fi2}, and~\ref{mainthm:fi3}. In \S \ref{ss:qordered}, we introduce a class of formal languages, the {\bf quasi-ordered languages}, and prove results about their Hilbert series. In \S \ref{sec:weighted-surj}, we discuss representations of the category $\FWS_{\Lambda}$ of finite weighted sets. Quasi-ordered languages are used to establish the main result about Hilbert series of representations of this category, and this result about Hilbert series is the key input to the proof of Theorem~\ref{mainthm:fs2}. In \S \ref{sec:FSG}, we study $\FS_G^\op$. We prove Theorems~\ref{mainthm:fs1} and~\ref{mainthm:fs2}, discuss the connection to $\Delta$-modules, and give some examples. 

\subsection{Notation}

\begin{itemize}
\item $\bN = \{0, 1, \dots\}$ denotes the set of non-negative integers.
\item For $n \in \bN$, define $[n] = \{1, \dots, n\}$ with the convention $[0] = \emptyset$.
\item $\widehat{\Sym}$ is the completion of the symmetric algebra, so $\SYM(V) = \prod_{d \ge 0} \Sym^d (V)$.
\item For a partition $\lambda$, we let $\bM_\lambda$ denote the Specht module, which is a representation of the symmetric group $S_n$ with $n=|\lambda|$ (and is irreducible in characteristic $0$). See \cite[Chapter 4]{kerber} for details; there it is denoted $[\lambda]$.
\end{itemize}

\subsection*{Acknowledgements}

We thank Jeremy Miller for explaining the construction in \S\ref{ss:homotopy-conf}. We also thank Aur\'elien Djament, Nathalie Wahl, and a referee for helpful comments.

\section{Background} 
\label{sec:background}

\subsection{Representations of categories}
\label{ss:rep}

In this section, we recall the main combinatorial preliminaries that we need from \cite{catgb}. We also prove some additional results.

Let $\cC$ be an essentially small category. We denote by $\vert \cC \vert$ the set of isomorphism classes in $\cC$. We say that $\cC$ is {\bf directed} if every self-map in $\cC$ is the identity. If $\cC$ is directed, then $\vert \cC \vert$ is naturally a poset by defining $x \le y$ if there exists a morphism $x \to y$. 

Fix a nonzero ring $\bk$ (not necessarily commutative) and let $\Mod_{\bk}$ denote the category of left $\bk$-modules. A {\bf representation} of $\cC$ over $\bk$ (or a {\bf $\bk[\cC]$-module}) is a functor $\cC \to \Mod_{\bk}$. A map of $\cC$-modules is a natural transformation. We write $\Rep_{\bk}(\cC)$ for the category of representations, which is abelian. 

Let $x$ be an object of $\cC$. Define a representation $P_x$ of $\cC$ by $P_x(y)=\bk[\Hom(x,y)]$, i.e., $P_x(y)$ is the free left $\bk$-module with basis $\Hom(x,y)$. If $M$ is another representation then $\Hom(P_x, M)=M(x)$. This shows that $\Hom(P_x,-)$ is an exact functor, and so $P_x$ is a projective object of $\Rep_{\bk}(\cC)$. We call it the {\bf principal projective} at $x$. A $\cC$-module is {\bf finitely generated} if it is a quotient of a finite direct sum of principal projective objects. 

An object of $\Rep_{\bk}(\cC)$ is {\bf noetherian} if every ascending chain of subobjects stabilizes; this is equivalent to every subrepresentation being finitely generated. The category $\Rep_{\bk}(\cC)$ is {\bf noetherian} if every finitely generated object in it is.

Let $\Phi \colon \cC' \to \cC$ be a functor. This gives a pullback functor on representations 
\[
\Phi^* \colon \Rep_{\bk}(\cC) \to \Rep_{\bk}(\cC')
\]
defined by $\Phi^* M = M \circ \Phi$. We are interested in how $\Phi^*$ interacts with finiteness properties of representations. The following definition (which is \cite[Definition~3.2.1]{catgb}) is of central importance to this.

\begin{definition} \label{def:propF}
We say that $\Phi$ satisfies {\bf property (F)} if the following condition holds: given any object $x$ of $\cC$ there exist finitely many objects $y_1, \ldots, y_n$ of $\cC'$ and morphisms $f_i \colon x \to \Phi(y_i)$ in $\cC$ such that for any object $y$ of $\cC'$ and any morphism $f \colon x \to \Phi(y)$ in $\cC$, there exists a morphism $g \colon y_i \to y$ in $\cC'$ such that $f=\Phi(g) \circ f_i$.
\end{definition}

The following result shows why property~(F) is so useful. The next three statements are \cite[Propositions 3.2.3, 3.2.4, Corollary 3.2.5]{catgb}:

\begin{proposition} \label{propFfg}
The functor $\Phi$ satisfies property~{\rm (F)} if and only if $\Phi^*$ takes finitely generated objects of $\Rep_{\bk}(\cC)$ to finitely generated objects of $\Rep_{\bk}(\cC')$.
\end{proposition}

Recall that $\Phi$ is essentially surjective if every object of $\cC$ is isomorphic to $\Phi(x)$ for some object $x \in \cC'$.

\begin{proposition} \label{pullback-fg}
Suppose that $\Phi$ is essentially surjective. Let $M$ be an object of $\Rep_{\bk}(\cC)$ such that $\Phi^*(M)$ is finitely generated {\rm (}resp.\ noetherian{\rm )}. Then $M$ is finitely generated {\rm (}resp.\ noetherian{\rm )}.
\end{proposition}

\begin{proposition} \label{pullback-noeth}
Suppose that $\Rep_{\bk}(\cC')$ is noetherian and $\Phi$ is essentially surjective and satisfies property~{\rm (F)}. Then $\Rep_{\bk}(\cC)$ is noetherian.
\end{proposition}

\begin{proposition} \label{propFcomp}
Consider functors $\Phi \colon \cC_1 \to \cC_2$ and $\Psi \colon \cC_2 \to \cC_3$. 
\begin{enumerate}[\indent \rm (a)]
\item If $\Phi, \Psi$ satisfy property~{\rm (F)}, then the composition $\Psi \circ \Phi$ satisfies property~{\rm (F)}.
\item If $\Phi$ is essentially surjective and $\Psi \circ \Phi$ satisfies property~{\rm (F)}, then $\Psi$ satisfies property~{\rm (F)}.
\end{enumerate}
\end{proposition}

\begin{proof}
This follows immediately from the previous three propositions.
\end{proof}

A {\bf norm} on $\cC$ is a function $\nu \colon \vert \cC \vert \to \bN^I$, where $I$ is a finite set. A {\bf normed category} is a category equipped with a norm. Fix a category $\cC$ with a norm $\nu$ with values in $\bN^I$; let $\{t_i\}_{i \in I}$ be indeterminates. Let $M$ be a representation of $\cC$ over a field $\bk$. We define the {\bf Hilbert series} of $M$ as
\begin{displaymath}
\rH_{M, \nu}(\bt) = \sum_{x \in \vert \cC \vert} \dim_{\bk}{M(x)} \cdot \bt^{\nu(x)},
\end{displaymath}
when this makes sense. We omit the norm $\nu$ from the notation when possible.

\subsection{Gr\"obner bases} \label{ss:grobner-bases}

A poset $X$ is {\bf noetherian} if for every sequence $x_1, x_2, \dots$ of elements in $X$, there exists $i < j$ such that $x_i \le x_j$. See \cite[\S 2]{catgb} for basic facts.

For an object $x$, let $S_x \colon \cC \to \Set$ be the functor given by $S_x(y)=\Hom(x, y)$. Note that $P_x=\bk[S_x]$. An {\bf ordering} on $S_x$ is a choice of well-order on $S_x(y)$, for each $y \in \cC$, such that for every morphism $y \to z$ in $\cC$ the induced map $S_x(y) \to S_x(z)$ is strictly order-preserving. We write $\preceq$ for an ordering; $S_x$ is {\bf orderable} if it admits an ordering.

Set 
\[
\wt{S}_x = \coprod_{y \in \cC} S_x(y).
\]
Given $f \in S_x(y)$ and $g \in S_x(z)$, define $f \le g$ if there exists $h \colon y \to z$ with $h_*(f)=g$. Define an equivalence relation on $\wt{S}_x$ by $f \sim g$ if $f \le g$ and $g \le f$. The poset $\vert S_x \vert$ is the quotient of $\wt{S}_x$ by $\sim$, with the induced partial order.

\begin{definition} \label{def:grobner-cat}
We say that $\cC$ is {\bf Gr\"obner} if, for all objects $x$, the functor $S_x$ is orderable and the poset $\vert S_x \vert$ is noetherian. We say that $\cC$ is {\bf quasi-Gr\"obner} if there exists a Gr\"obner category $\cC'$ and an essentially surjective functor $\cC' \to \cC$ satisfying property~(F).
\end{definition}

The following statement is \cite[Theorem 4.3.2]{catgb}:

\begin{theorem} \label{grobnoeth}
Let $\cC$ be a quasi-Gr\"obner category. Then for any left-noetherian ring $\bk$, the category $\Rep_{\bk}(\cC)$ is noetherian.
\end{theorem}

The following results follow easily from the definitions.

\begin{proposition} \label{potgrob}
Suppose that $\Phi \colon \cC' \to \cC$ is an essentially surjective functor satisfying property~{\rm (F)} and $\cC'$ is quasi-Gr\"obner. Then $\cC$ is quasi-Gr\"obner.
\end{proposition}

\begin{proposition} \label{grobprod}
The cartesian product of finitely many {\rm (}quasi-{\rm )}Gr\"obner categories is {\rm (}quasi-{\rm )}Gr\"obner.
\end{proposition}

\section{Categories of $G$-injections} \label{sec:fig}

\subsection{Finite groups}

In this section, we assume that $G$ is finite. 

Define a functor $\Phi \colon \FA \to \FA_G$ that sends a set to itself and a function $f \colon x \to y$ to $(f,1)$ where $1 \colon x \to G$ is the constant map sending every element to the identity of $G$. We also use $\Phi$ to denote the restriction $\Phi \colon \FI \to \FI_G$.

We have the following basic property about representations of $\FI_G$:

\begin{proposition} \label{prop:FIG-propertyF}
The functors $\Phi \colon \FA \to \FA_G$ and $\Phi \colon \FI \to \FI_G$  satisfy property~{\rm (F)}. In particular, $\FA_G$ and $\FI_G$ are quasi-Gr\"obner categories, and if $\bk$ is left-noetherian then $\Rep_\bk(\FA_G)$ and $\Rep_\bk(\FI_G)$ are noetherian.
\end{proposition}

\begin{proof}
To see that $\Phi$ satisfies property~(F), suppose $x$ has size $n$, set $y_1, \dots, y_{n^{|G|}}$ all equal to $x$ and let $f_1, \dots, f_{n^{|G|}}$ correspond to all automorphisms of $x$ in $\FA_G$ of the form $(1, g)$ under some enumeration. The categories $\FI$ and $\FA$ are quasi-Gr\"obner by \cite[Theorems 7.1.4, 7.3.4]{catgb}, so the last part follows from Proposition~\ref{potgrob}.
\end{proof}

Corollary~\ref{cor:polycyclic} improves the noetherian conclusion by allowing $G$ to be any polycyclic-by-finite group.

\begin{remark}
Recall from \cite[\S 7.1]{catgb} that $\FI_d$ is the category of finite sets where a morphism $S \to T$ is an injection $f \colon S \to T$ and a function $g \colon T \setminus f(S) \to \{1, \dots, d\}$ with composition defined in the evident way. Define a category $\FI_{d,G}$ of finite sets whose morphisms are pairs $(f,\sigma)$ where $f$ is a decorated injection as in the definition of $\FI_d$ and $\sigma$ is as in the definition of $\FI_G$. As above, there is a natural functor $\FI_d \to \FI_{d,G}$ satisfying property~(F). 
\end{remark}

The category $\Rep_{\bk}(\FI_G)$ only depends on $\Rep_{\bk}(G)$ as an abelian category equipped with the extra structure of the invariants functor $\Rep_{\bk}(G) \to \Mod_{\bk}$. Thus $\Rep_{\bk}(\FI_G)$ ``sees'' very little of $G$. We can sometimes be more explicit. Let $\FB$ be the groupoid of finite sets (maps are bijections).

\begin{proposition} \label{FIGrep}
Suppose that $\bk$ is a splitting field for $G$ in which $\# G$ is invertible. Then representations of $\FI_G$ are equivalent to representations of $\FI \times \FB^r$, where $r$ is the number of non-trivial irreducible representations of $G$ over $\bk$.
\end{proposition}

\begin{proof}
Let $V_1, \ldots, V_r$ be the non-trivial irreducible representations of $G$, and let $V_0$ be the trivial representation. Suppose that $M$ is an $\FI_G$-module. Since $\bk$ is a splitting field for $G$, $\bk[G]$ is isomorphic to a direct product of matrix algebras; since $\bk[G^S] \cong \bk[G]^{\otimes S}$, we see that $\bk$ is also a splitting field for all $G^S$. So we can decompose $M(S)$ into isotypic pieces for the action of $G^S$:
\begin{equation} \label{eq5}
M(S) = \bigoplus_{S=S_0 \amalg \cdots \amalg S_r} N(S_0, \ldots, S_r) \otimes (V_0^{\boxtimes S_0} \boxtimes \cdots \boxtimes V_r^{\boxtimes S_r} ),
\end{equation}
where $N$ is a multiplicity space. Suppose now that $f \colon S \to T$ is an injection. To build a morphism in $\FI_G$ we must also choose a function $\sigma \colon S \to G$. However, if $\sigma$ and $\sigma'$ are two choices then $(f,\sigma)$ and $(f,\sigma')$ differ by an element of $\Aut(T)$, namely an automorphism of the form $(\id_T, \tau)$ where $\tau$ restricts to $\sigma' \sigma^{-1}$ on $S$. Thus it suffices to record the action of $(f, 1)$. Note that if $\tau \colon T \to G$ restricts to 1 on $S$ then $(\id_T, \tau) (f,1)=(f,1)$. It follows that $(f,1)$ must map $M(S)$ into the $G^{T \setminus S}$-invariants of $M(T)$. In other words, under the above decomposition, $(f,1)$ induces a linear map
\begin{displaymath}
N(S_0, S_1, \ldots, S_r) \to N(f(S_0) \amalg (T \setminus f(S)), f(S_1), \ldots, f(S_r)).
\end{displaymath}
Thus, associated to $M$ we have built a representation $N$ of $\FI \times \FB^r$. The above discussion makes clear that no information is lost in passing from $M$ to $N$, and so this is a fully faithful construction. The inverse construction is defined by the formula \eqref{eq5}.
\end{proof}

\begin{remark}
By Proposition~\ref{FIGrep}, an $\FI_G$-module can be thought of as a sequence $(M_{\bn})_{\bn \in \bN^r}$, where each $M_{\bn}$ is an $\FI$-module equipped with an action of $S_{n_1} \times \cdots \times S_{n_r}$. There are no transition maps, so in a finitely generated $\FI_G$-module, all but finitely many of the $M_{\bn}$ are zero. Thus, at least in good characteristic, $\FI_G$-modules are not much different from $\FI$-modules, and essentially any result about $\FI$-modules (e.g., noetherianity) carries over to $\FI_G$-modules.

There are some similarities between Proposition~\ref{FIGrep} and the results of \cite[\S 6]{macdonald-wr}.
\end{remark}

We need one last thing for our application in the next section. Suppose that $\bk$ is commutative. Given $\FI_G$-modules $M$ and $N$, their {\bf pointwise tensor product}, denoted $M \odot N$ is the $\FI_G$-module given by $S \mapsto M(S) \otimes_\bk N(S)$ and $f \mapsto M(f) \otimes N(f)$ for morphisms $f$.

\begin{proposition} \label{prop:segre-FIG}
For any commutative ring $\bk$, the pointwise tensor product of two finitely generated $\FI_G$-modules is finitely generated.
\end{proposition}

\begin{proof}
By \cite[Proposition 3.3.2]{catgb}, it suffices to prove that the diagonal functor 
\[
\Delta \colon \FI_G \to \FI_G \times \FI_G
\]
satisfies property (F). To see this, let $x$ and $x'$ be finite sets and consider $G$-injections 
\[
(f,\sigma) \colon x \to y \text{ and } (f',\sigma') \colon x' \to y.
\]
Let $\ol{y} = f(x) \cup f'(x')$ and define a morphism $(g,1) \colon \ol{y} \to y$ where $g \colon \ol{y} \to y$ is the inclusion and $1$ denotes the constant map $\ol{y} \to G$. Also define $\ol{f} \colon x \to \ol{y}$ and $\ol{f}' \colon x' \to \ol{y}$ to be the maps induced by $f$ and $f'$. Then $(f, \sigma) = (g,1) \circ (\ol{f}, \sigma)$ and $(f', \sigma') = (g,1) \circ (\ol{f}', \sigma')$. Since $\#\ol{y} \le \#x + \#x'$ and $G$ is finite, there are only finitely many choices for $\ol{y}$, $(\ol{f}, \sigma)$, and $(\ol{f}', \sigma')$ up to isomorphism and taking all of these choices shows that $\Delta$ satisfies property (F).
\end{proof}

\subsection{Wreath product version of Murnaghan's stability theorem} \label{sec:wreath-stab}

Using Proposition~\ref{FIGrep} and following \cite[\S 3.4]{fimodules}, we can deduce a generalization of Murnaghan's stability theorem for tensor products of representations of symmetric groups to arbitrary wreath products. For this section, set $\bk = \bC$, and assume that $G$ is finite.

For basics on complex representations of wreath products of finite groups, see \cite[Chapter 5]{kerber}. Let $R(G)$ be the set of isomorphism classes of irreducible representations of $G$. The irreducible representations $V(\ul{\lambda})$ of the wreath product $S_n \ltimes G^n$ are indexed by partition-valued functions $\ul{\lambda}$ on $R(G)$ such that $|\ul{\lambda}| = \sum_{x \in R(G)} |\ul{\lambda}(x)| = n$. 

\begin{remark} \label{rmk:wr-translate}
To relate this with Proposition~\ref{FIGrep}, we make the following observations. First, given a representation $V$ of $G$ and a representation $W$ of $S_n$, the vector space $V^{\otimes n} \otimes W$ carries an action of $S_n \ltimes G^n$: for $v_i \in V$, $w \in W$, $(g_1, \dots, g_n) \in G^n$, and $\sigma \in S_n$, we have
\begin{align*}
(g_1, \dots, g_n) \cdot (v_1 \otimes \cdots \otimes v_n \otimes w) &= g_1v_1 \otimes \cdots \otimes g_nv_n \otimes w\\
\sigma \cdot (v_1 \otimes \cdots \otimes v_n \otimes w) &= v_{\sigma^{-1}(1)} \otimes \cdots \otimes v_{\sigma^{-1}(n)} \otimes \sigma w.
\end{align*}
This extends to an action of $S_n \ltimes G^n$ on $V^{\otimes n} \otimes W$; write this representation as $V \wr W$. The representation $V(\ul{\lambda})$ can be constructed as follows. For each $V \in R(G)$, let $n_V = |\ul{\lambda}(V)|$; $\bigotimes_V V \wr \bM_{\ul{\lambda}(V)}$ is a representation of $\prod_V S_{n_V} \ltimes G^{n_V}$, and this group can be identified with a subgroup of $S_n \ltimes G^n$. The induction is $V(\ul{\lambda})$. 

If we think of $V(\ul{\lambda})$ as a $\FI_G$-module concentrated in degree $n$, then we see from the description just given that, under the bijection in Proposition~\ref{FIGrep}, $V(\ul{\lambda})$ corresponds to the tensor product of the $\bM_{\ul{\lambda}(V)}$, thought of as an $\FI \times \FB^r$-module.
\end{remark}

Given a partition $\lambda$ and a positive integer $n$, define $\lambda[n] = (n-|\lambda|, \lambda_1, \lambda_2, \dots)$ (we assume that $n$ is large enough so that this is a partition). Given any partition-valued function $\ul{\lambda}$ on $R(G)$, let $\ul{\lambda}[n]$ denote the function which agrees with $\ul{\lambda}$ on all non-trivial representations and is $\lambda(1_G)[n]$ on the trivial representation $1_G$.

\begin{theorem} \label{thm:wr-stab}
Fix partition-valued functions $\ul{\lambda}, \ul{\mu}, \ul{\nu}$ on $R(G)$. The multiplicity of $V(\ul{\nu}[n])$ in $V(\ul{\lambda}[n]) \otimes V(\ul{\mu}[n])$ as representations of $S_n \ltimes G^{n}$ is independent of $n$ for $n \gg 0$.
\end{theorem}

\begin{proof}
We first recall a few facts about $\FI$-modules from \cite{symc1} (that paper is not written in the language of $\FI$-modules, but see \cite[Proposition 1.3.5]{symc1}). Let $\Mod_\FI^{\tors}$ be the category of finitely generated torsion $\FI$-modules (for finitely generated modules, torsion is equivalent to being nonzero only in finitely many degrees) and let $\Mod_\FI$ be the category of finitely generated $\FI$-modules. Then every object in $\Mod_K := \Mod_\FI / \Mod_{\FI}^\tors$ has finite length \cite[Corollary 2.2.6]{symc1}. Furthermore, the pointwise tensor product of two finitely generated $\FI$-modules is also finitely generated \cite[Proposition 2.3.6]{fimodules}, and hence the pointwise tensor product of two finite length objects in $\Mod_K$ still has finite length.

For every partition $\lambda$, there is an $\FI$-module $L_\lambda^0$ generated in degree $|\lambda|$ whose value on a set of size $n$ is the representation $\bM_{\lambda[n]}$ (see \cite[\S 2.2]{symc1}). Let $L_\lambda$ be the image of $L_\lambda^0$ in $\Mod_K$; $L_\lambda$ is a simple object, and all simple objects are of this form \cite[Corollary 2.2.7]{symc1}.

A representation $\bM_\mu$ is naturally a functor on $\FB$ and functors on $\FI \times \FB^r$ can be built by taking the tensor product of functors on $\FI$ and $r$ functors on $\FB$. For every partition-valued function $\ul{\lambda}$ on $R(G)$, define a functor $L_{\ul{\lambda}}^0$ on $\FI \times \FB^r$ which is the tensor product of $L_{\lambda(1_G)}^0$ on $\FI$ and $\bM_{\lambda(V)}$ on the copy of $\FB$ which is indexed by $V$. From Remark~\ref{rmk:wr-translate}, using the equivalence in Proposition~\ref{FIGrep}, this corresponds to a functor on $\FI_G$ whose value on a set of size $n$ is $V(\ul{\lambda}[n])$ if $n$ is large enough for $\ul{\lambda}[n]$ to be a partition-valued function (and $0$ otherwise). 

Write $\Mod_{K,G} := \Mod_{\FI_G} / \Mod_{\FI_G}^\tors$. Using the facts for $\FI$ and $\Mod_K$ recalled above, every object of $\Mod_{K,G}$ has finite length and the simple objects are of the form $L_{\ul{\lambda}}$. Furthermore, the pointwise tensor product preserves finite generation in $\FI_G$ (Proposition~\ref{prop:segre-FIG}), so $L_{\ul{\lambda}} \boxtimes L_{\ul{\mu}}$ has a finite filtration by modules of the form $L_{\ul{\nu}}$ in $\Mod_{K,G}$. Finally, we note that the modules in $\Mod_{\FI_G}^\tors$ are concentrated in finitely many degrees. So the filtration encodes the usual tensor product for large enough degree. This implies the desired stability result.
\end{proof}

\begin{remark}
When $G$ is trivial, the proof of Theorem~\ref{thm:wr-stab} generalizes the one given in \cite[\S 3.4]{fimodules} and when $G = \bZ/2$, Theorem~\ref{thm:wr-stab} is proven in \cite[Theorem 5.3]{wilson}. 
\end{remark}

\subsection{Noetherianity}

In this section we prove Theorem~\ref{mainthm:fi1}. Define $\OI_G$ to be the category whose objects are finite ordered sets and where morphisms are pairs $(f, \rho)$ as in $\FI_G$, except that $f$ is required to be order-preserving. We first prove:

\begin{theorem} \label{oig:noeth}
Suppose that $\bk[G^n]$ is left-noetherian for all $n \ge 0$. Then $\Rep_{\bk}(\OI_G)$ is noetherian.
\end{theorem}

\begin{proof}
It suffices to show that the principal projective $P_n$ of $\OI_G$ at $[n]$ is noetherian. As a $\bk$-module, we have
\begin{displaymath}
P_n([m])=\bigoplus_{f \colon [n] \to [m],\ a \in G^n} \bk e_{(f,a)}
\end{displaymath}
where $(f,a)$ runs over $\Hom_{\OI_G}([n], [m])$ and $e_{(f,a)}$ is a basis vector. Set $R = \bk[G^n]$. We give $P_n([m])$ the structure of a left $R$-module by $be_{(f,a)}=e_{(f,ba)}$. Let $e_f=e_{(f,1)}$. Then the vectors $e_f$'s form a basis for $P_n([m])$ as a left $R$-module. Suppose that $g \colon [m] \to [\ell]$ is a morphism in $\OI$ and $b \in G^n$. Then
\begin{displaymath}
g_*(be_{(f,a)})=g_*(e_{(f,ba)})=e_{(gf,ba)}=be_{(gf,a)}.
\end{displaymath}
Thus the $\OI$-module structure on $P_n$ is compatible with the left $R$-module structure. Now, $P_n([m])$ also has a left action of $\Aut_{\OI_G}([m])=G^m$. If $b \in G^m$ then $b e_{(f,a)}=e_{(f,f^*(b)a)}$ where $f^* \colon G^m \to G^n$ is the pullback map. We can thus write $b a e_f=f^*(b) a e_f$ for $a \in R$.

In general, $\bk[\OI_G]$-submodules of $P_n$ need not be $R[\OI]$-submodules (see Remark~\ref{rmk:emulate}), but this is true for ``monomial'' subrepresentations, which is enough to prove the theorem, as we now explain.

Following \cite[\S\S 4.2, 7.1]{catgb}, we order the set of order-preserving injections $[n] \to [m]$ ($n$ fixed, $m$ varying) using the lexicographic order. Given $x = \sum_i b_i e_{f_i} \in P_n([m])$ where $b_i \in R$ is non-zero and $f_i \in \hom_\OI([n],[m])$, with the $f_i$ distinct, define $\init(x)$ to be $b_i e_{f_i}$ where $f_i$ is lexicographically maximal among the $f$'s occurring in the sum.

Let $M$ be a $\bk[\OI_G]$-submodule of $P_n$. Define the initial submodule $\init(M)$ of $M$ by setting $\init(M)(S)$ to be the $\bk$-span of $\{ \init(x) \mid x \in M(S)\}$. By \cite[\S 4.2]{catgb}, $\init(M)$ is a $\bk[\OI]$-submodule of $P_n$. We claim that $\init(M)$ is also an $R$-submodule. Indeed, suppose that $be_f \in \init(M)([m])$. Let $x=\sum_i b_i e_{f_i}$ be an element of $M([m])$ with $\init(x)=be_f$. We assume that $b_1e_{f_1}=be_f$, and that the $f_i$ are distinct. Pick $a \in G^n$. The forgetful map $f_1^* \colon G^m \to G^n$ is surjective, so  choose $\wt{a} \in G^m$ with $f_1^*(\wt{a})=a$. We have
\begin{displaymath}
\wt{a} x = a b_1 e_{f_1} + \sum_{i>1} f_i^*(\wt{a}) b_i e_{f_i}.
\end{displaymath}
Since $a$ is a unit of $R$, it follows that $ab_1 \ne 0$, and so $\init(\wt{a} x)=a b_1 e_{f_1}$. We have $\wt{a} x \in M([m])$, since it is obtained from $x$ using the $\OI_G$-module structure, and so we see that $a b_1 e_{f_1}=a \cdot \init(x)$ belongs to $\init(M)([m])$, proving the claim.

Now suppose that $M_1 \subseteq M_2 \subseteq \cdots$ is an ascending chain of $\bk[\OI_G]$-submodules of $P_n$. Then $\init(M_1) \subseteq \init(M_2) \subseteq \cdots$ is an ascending chain of $R[\OI]$-submodules of $P_n$, and therefore stabilizes since $\Rep_R(\OI)$ is noetherian \cite[Theorem~7.1.1]{catgb}. It follows from a standard Gr\"obner basis argument (see \cite[Proposition 4.2.2]{catgb}) that the chain $M_{\bullet}$ stabilizes, and so $P_n$ is noetherian.
\end{proof}

\begin{remark} \label{rmk:emulate}
If $M$ is a $\bk[\OI_G]$-submodule of $P_n$ then it need not be closed under the action of $R$ by pre-composition: that is, there is no way to emulate pre-composition with automorphisms of $[n]$ using post-composition with other morphisms. We do not prove this (as it is unnecessary to do so), but give an illustrative example. The group $P_n([m])$ has the structure of a $\bk[G^m]$-module, since $G^m \subset \Aut_{\OI_G}([m])$. If $a \in \bk[G^m]$ and $x=\sum b_f e_f$ is in $P_n([m])$, with $b_f \in \bk[G^n]$, then $ax=\sum f^*(a) b_f e_f$, where $f^* \colon \bk[G^m] \to \bk[G^n]$ is induced by $f^* \colon G^m \to G^n$. Given $\ol{a} \in R$, it is not generally possible to pick $a \in \bk[G^m]$ such that $f^*(a)=\ol{a}$ for all $f$, and so one cannot emulate the $R$-structure using the $\bk[G^m]$-structure.
\end{remark}

Before proving Theorem~\ref{mainthm:fi1}, we require two lemmas.

\begin{lemma} \label{oigF}
The forgetful functor $\Phi \colon \OI_G \to \FI_G$ satisfies property~{\rm (F)}.
\end{lemma}

\begin{proof}
Let $x$ be an object of $\FI_G$ and $y$ an object of $\OI_G$. Choose a total order on $x$, and write $\ol{x}$ for the corresponding object of $\OI_G$, so that $x=\Phi(\ol{x})$. Consider a morphism $f=(f_0,\rho) \colon x \to \Phi(y)$ in $\FI_G$. We can factor $f_0$ as $x \stackrel{\sigma}{\to} x \stackrel{g_0}{\to} y$, where $\sigma$ is a permutation and $g_0$ is order-preserving. Let $g=(g_0, \rho \circ \sigma^{-1})$, so that $g$ defines a morphism $\ol{x} \to \ol{y}$ in $\OI_G$. Let $\wt{\sigma}=(\sigma,1)$, where $1$ is the constant function $x \to G$, so that $\wt{\sigma} \colon x \to x$ is a morphism in $\FI_G$. Then $f=\Phi(g) \circ \wt{\sigma}$. Since there are only finitely many choices for $\wt{\sigma}$, the result follows.
\end{proof}

\begin{lemma} \label{twisted-noeth}
Let $R$ be a ring and let $H$ be a group acting on $R$. Suppose that the twisted group algebra $R[H]$ is left-noetherian. Then $R$ is left-noetherian.
\end{lemma}

Recall that the twisted group algebra is the set of $R$-linear combinations of elements in $H$ with the multiplication $(rh)(r'h') = rh(r') hh'$.

\begin{proof}
Let $M$ be a left $R$-module. For $h \in H$, let $M^{h}$ be the left $R$-module with underlying abelian group $M$ on which $R$ acts by $x \cdot m=h(x) m$. Let $\wt{M} = \bigoplus_{h \in H} M^{h}$. Then $\wt{M}$ is naturally a left $R[H]$-module. Suppose now that $I_1 \subseteq I_2 \subseteq \cdots$ is an ascending chain of left ideals of $R$. Then $\wt{I}_1 \subseteq \wt{I}_2 \subseteq \cdots$ is an ascending chain of left ideals of $R[H]$, and therefore stabilizes. This clearly implies that the original chain stabilizes as well, and so $R$ is left-noetherian.
\end{proof}

\begin{proof}[Proof of Theorem~\ref{mainthm:fi1}]
First suppose that $\bk[G^n]$ is left-noetherian for all $n \ge 0$. Then $\Rep_{\bk}(\OI_G)$ is noetherian by Theorem~\ref{oig:noeth}. Thus by Lemma~\ref{oigF} and Proposition~\ref{pullback-noeth} we find that $\Rep_{\bk}(\FI_G)$ is noetherian. 
Conversely, suppose that $\Rep_{\bk}(\FI_G)$ is noetherian. Let $R=\bk[G^n]$ and let $R[S_n] \cong \bk[S_n \ltimes G^n]$ be the twisted group algebra ($S_n$ acting on $G^n$ by permutations). Then a left $R[S_n]$-module is the same as a $\bk[\FI_G]$-module supported at $n$. It follows that $R[S_n]$ is left-noetherian, and so $R$ is left-noetherian by Lemma~\ref{twisted-noeth}.
\end{proof}

\subsection{Twisted homological stability for wreath products} \label{ss:twisted-stab}

Let $G$ be a group and $\bk$ be a commutative ring and let $M \colon \FI_G \to \Mod_\bk$ be an $\FI_G$-module. Let $M_n = M([n])$. Consider the morphism $[n] \to [n+1]$ consisting of the injection defined by $i \mapsto i$ and the trivial $G$-map $[n]\to G$ sending every element to the identity. This gives a $\bk[S_n \ltimes G^n]$-equivariant map $f_n \colon M_n \to M_{n+1}$.

\begin{theorem} \label{thm:wr-twisted-stab}
Assume that $G$ is a polycyclic-by-finite group and that $\bk$ is noetherian, and let $M$ be a finitely generated $\FI_G$-module. Then for each $i \ge 0$, the map
\[
{f_n}_\ast \colon \rH_i(S_n \ltimes G^n; M_n) \to \rH_i(S_{n+1} \ltimes G^{n+1}; M_{n+1})
\]
is an isomorphism for $n \gg 0$.
\end{theorem}

\begin{proof}
We apply \cite[Theorem 4.2]{putman-sam}. First, $\FI_G$ is a complemented category where the monoidal structure is given by disjoint union of sets, and the generator is the one element set. The two conditions of that theorem are: (1) the category of $\FI_G$-modules is noetherian, and (2) the isomorphism above holds for $n \gg 0$ when $M$ is the trivial $\FI_G$-module. The first point is Corollary~\ref{cor:polycyclic} and the second point is \cite[Proposition 1.6]{hatcher-wahl}.
\end{proof}

Define a shift functor $\Sigma$ on $\FI_G$-modules by $(\Sigma M)(S) = M(S \amalg \{*\})$. In \cite[Definition 4.10]{wahl}, the following definition is introduced (to avoid ambiguity, we use the phrase $\Sigma$-degree instead of just degree). First, a functor which takes the value $0$ except at finitely many isomorphism classes of objects has $\Sigma$-degree $-1$ and, in general, $F$ has $\Sigma$-degree $\le r$ if the kernel and cokernel of the natural map $F \to \Sigma F$ have $\Sigma$-degree $\le r-1$. Then \cite[Theorem 4.20]{wahl} proves Theorem~\ref{thm:wr-twisted-stab} for any $G$ under the assumption that $M$ has finite $\Sigma$-degree (and gives bounds for when stability occurs in terms of the $\Sigma$-degree). 

So it is natural to ask if a finitely generated $\FI_G$-module necessarily has finite $\Sigma$-degree. We do not know if this is true, but will now prove it when $G$ is finite.

\begin{proposition} \label{prop:finiteG-polynomial}
If $G$ is finite, then a finitely generated $\FI_G$-module has finite $\Sigma$-degree.
\end{proposition}

\begin{proof}
Let $M$ be an $\FI_G$-module. The kernel of $M \to \Sigma M$ is torsion, i.e., all non-invertible morphisms act by zero. If $M$ is finitely generated, then the same is true for the kernel (Proposition~\ref{prop:FIG-propertyF}), and hence it is supported in finitely many degrees. It is easy to see that a torsion module  concentrated in degrees $\le d$ has $\Sigma$-degree $\le d$, so from now on, we only need to focus on the cokernel of $M \to \Sigma M$, which we denote $\Delta M$.

Note that $\Delta$ is right-exact. Also, 
\begin{align*}
\Sigma P_S &= P_S \oplus \bigoplus_{S' \subset S} \bk[G] \otimes_{\bk} P_{S'}\\
\Delta P_S &= \bigoplus_{S' \subset S} \bk[G] \otimes_{\bk} P_{S'},
\end{align*}
where $\# S'=\# S-1$ in both sums. Since $G$ is finite, we observe that if $M$ is finitely generated in degree $\le d$, then $\Delta M$ is finitely generated in degree $\le d-1$ (here we use that $G$ is finite). By induction, we are done.
\end{proof}

\begin{remark}
\begin{enumerate}[(a)]
  \item Proposition~\ref{prop:finiteG-polynomial} and the discussion above shows that Theorem~\ref{thm:wr-twisted-stab} follows from \cite[Theorem 4.20]{wahl} when $G$ is finite. When $G$ is trivial, we have been informed that Theorem~\ref{thm:wr-twisted-stab} is contained in work of Betley \cite{betley} and Church \cite{church} also with bounds for when stability occurs. By Proposition~\ref{prop:FIG-propertyF}, when $G$ is finite, the result also follows from this work. Our result is new when $G$ is infinite. 

\item Using \cite[Theorem B]{ramos}, the proof of Proposition~\ref{prop:finiteG-polynomial} can be adapted to work for any group $G$ if we replace finitely generated by the condition that it is presented in finite degree. In particular, if $G$ is polycyclic-by-finite, the statement of Proposition~\ref{prop:finiteG-polynomial} remains valid. Along the same lines, by \cite{palmer}, the result also holds for any $G$ in the special case when the $\FI_G$-module can be extended to an $\FI^\#_G$-module ($\FI^\#_G$ is a larger category where morphisms are partially defined $G$-injections between finite sets).
\qedhere
\end{enumerate}
\end{remark}

\subsection{Homotopy groups of configuration spaces} \label{ss:homotopy-conf}

Let $\FI(G)$ be the category of sets with a free $G$-action with finitely many orbits and injective $G$-equivariant maps. There is an equivalence $\FI_G \to \FI(G)$ defined by $S \mapsto S \times G$ (the $G$-action is $h \cdot (s,g) = (s,gh^{-1})$) and sending $(f,\sigma) \colon S \to T$ to the morphism $S \times G \to T \times G$ given by $(s,g) \mapsto (f(s), \sigma(s) g)$.

Let $M$ be a connected manifold with $\dim(M) \ge 3$. Fix $k \ge 2$ and set $G = \pi_1(M)$. Let $\wt{M}$ be the universal cover of $M$, so $G$ acts on $\wt{M}$ by deck transformations. Given a set $X$ with a free $G$-action, let $\Conf^G_X(\wt{M})$ be the space of injective $G$-equivariant maps. Also, for any set $S$, let $\Conf_S(M)$ be the space of injective maps. There is a natural map $\Conf_X^G(\wt{M}) \to \Conf_{X/G}(M)$ with fiber $G^{X/G}$. Since $\dim(M) \ge 3$, $G^{X/G} = \pi_1(\Conf_{X/G}(M))$ and hence $\Conf^G_X(\wt{M})$ is simply-connected.

Given an equivariant injective map of sets $X \to Y$ with free $G$-action, we have a forgetful map $\Conf^G_Y(\wt{M}) \to \Conf_X^G(\wt{M})$ and hence a map on homotopy groups $\pi_k(\Conf^G_Y(\wt{M})) \to \pi_k(\Conf_X^G(\wt{M}))$ (we have not chosen basepoints: each $\Conf^G_Z(\wt{M})$ is simply-connected, so there is a canonical isomorphism between the homotopy groups for any two choices of basepoint).

So we can define a functor $\FI(G)^{\rm op} \to \Mod_\bZ$ by $X \mapsto \pi_k(\Conf^G_Y(\wt{M}))$. Using the equivalence above, this is also a functor on $\FI_G^{\rm op}$. Now let $A$ be any abelian group. We define an $\FI_G$-module by $S \mapsto \hom(\pi_k(\Conf_{S \times G}^G(\wt{M})), A)$.

\begin{remark}
When $G = \pi_1(M)$ is trivial and $A = \bQ$, we get an $\FI$-module defined over $\bQ$, which is studied in \cite{kupers-miller}. In particular, they show that these are finitely generated $\FI$-modules. They prove a similar result when $A = \bZ$ and also for the functor $S \mapsto \ext^1_\bZ(\pi_k(\Conf_{S \times G}^G(\wt{M})), \bZ)$. It would be interesting to find more general conditions on $M$, $A$, and $k$ under which the above $\FI_G$-module is finitely generated.
\end{remark}

\section{Quasi-ordered languages} \label{ss:qordered}
Let $\Sigma$ be an {\bf alphabet} (i.e., a finite set) and let $\Sigma^\star$ be the free monoid generated by $\Sigma$. {\bf Words} are elements of $\Sigma^\star$ and {\bf languages} are subsets of $\Sigma^\star$. We write words as $w = w_1 \cdots w_n$ where $w_i \in \Sigma$; the length of the word is $n$. A subword of $w$ is a word of the form $w_{i_1} \cdots w_{i_k}$ where $1 \le i_1 < \cdots < i_k \le n$ ($k=0$ is allowed). We can also represent a word as a function $[n] \to \Sigma$ given by $i \mapsto w_i$; equivalently, we can write them as functions $I \to \Sigma$ where $I$ is ordered, this is convenient for defining subwords by restricting the function to $I \subset [n]$.

Let $\Lambda$ be a finite abelian group and let $\phi \colon \Sigma^\star \to \Lambda$ be a monoid homomorphism. Given a subset $S \subseteq \Lambda$, define 
\[
\Sigma^{\star}_{\phi,S} = \{w \in \Sigma^{\star} \mid \phi(w) \in S\}.
\]
We say that a language $\cL \subset \Sigma^{\star}$ is a {\bf congruence language} if it is of the form $\Sigma^{\star}_{\phi,S}$ for some $\Lambda$, $\phi$ and $S$. The {\bf modulus} of a congruence language is the exponent of the group $\Lambda$. (The {\bf exponent} of a group is the least common multiple of the orders of all of its elements.)

Let $F(\bt)$ be a formal power series in variables $\bt=(t_1, \ldots, t_r)$. An {\bf $N$-cyclotomic translate} of $F$ is a series of the form $F(\zeta_1 t_1, \ldots, \zeta_r t_r)$, where $\zeta_1, \ldots, \zeta_r$ are $N$th roots of unity.

\begin{lemma} \label{lem:cyc-trans}
Let $\Lambda$ be a finite abelian group of exponent $N$, let $S$ be a subset of $\Lambda$, and let $\psi \colon \bN^r \to \Lambda$ be a monoid homomorphism. Suppose that $F(\bt)=\sum_{\bn \in \bN^r} a_{\bn} \bt^{\bn}$ is a formal power series over $\bC$. Let $G(\bt) = \sum a_{\bn} \bt^{\bn}$, where the sum is over $\bn \in \bN^r$ such that $\psi(\bn) \in S$. Then $G(\bt)$ is a $\bQ(\zeta_N)$-linear combination of $N$-cyclotomic translates of $F(\bt)$.
\end{lemma}

\begin{proof}
We have $G(\bt)=\sum_{\bn \in \bN^r} \chi(\psi(\bn)) a_{\bn} \bt^{\bn}$, where $\chi \colon \Lambda \to \{0,1\}$ is the characteristic function of $S$. Now express $\chi$ as a $\bQ(\zeta_N)$-linear combination of characters of $\Lambda$.
\end{proof}

A {\bf norm} on $\Sigma$ is a monoid homomorphism $\nu \colon \Sigma^\star \to \bN^I$ for some set $I$ induced by a function $\Sigma \to I$ (i.e., if $x \in \Sigma$ then $\nu(x)$ is a standard basis vector). It is {\bf universal} if $\Sigma \to I$ is injective. A norm $\nu$ on a language $\cL \subseteq \Sigma^\star$ is the restriction of a norm on $\Sigma$ to $\cL$. 

\begin{proposition} \label{hilb-cong-intersect}
Let $\cL$ be a language on $\Sigma$ equipped with a universal norm $\nu$ with values in $\bN^I$, let $\cK$ be a congruence language on $\Sigma$ of modulus $N$, and let $\cL'=\cL \cap \cK$. Then $\rH_{\cL',\nu}(\bt)$ is a $\bQ(\zeta_N)$-linear combination of $N$-cyclotomic translates of $\rH_{\cL,\nu}(\bt)$.
\end{proposition}

\begin{proof}
Choose $\phi \colon \Sigma \to \Lambda$ and $S \subset \Lambda$ so that $\cK=\Sigma^{\star}_{\phi,S}$. Since $\nu$ is universal, the map $\phi \colon \Sigma^{\star} \to \Lambda$ can be factored as $\psi \circ \nu$, where $\psi \colon \bN^I \to \Lambda$ is a monoid homomorphism. Thus if $\rH_{\cL}(\bt)=\sum_{\bn \in \bN^I} a_{\bn} \bt^{\bn}$, then $\rH_{\cL'}(\bt)$ is obtained by simply discarding the terms for which $\psi(\bn) \not\in S$. The result now follows from Lemma~\ref{lem:cyc-trans}.
\end{proof}

Recall from \cite[\S 5.3]{catgb} that an {\bf ordered language} on $\Sigma$ is a language obtained from the singleton languages and the languages $\Pi^\star$ for $\Pi \subseteq \Sigma$, using finite union and concatenation. 
A {\bf quasi-ordered language} (of modulus $N$) is the intersection of an ordered language and a congruence language (of modulus $N$); we say that the language is ${\rm QO}_N$. The class of quasi-ordered languages is not closed under unions, intersections, or concatenations, but here is a positive result (we will not use regular languages, so see \cite[\S 5.2]{catgb} for the definition):

\begin{proposition}
\begin{enumerate}[\rm (a)]
\item Congruence languages are regular.
\item Quasi-ordered languages are regular. 
\end{enumerate}
\end{proposition}

\begin{proof}
(a) Pick an isomorphism $\Lambda \cong \bZ/n_1 \times \cdots \times \bZ/n_r$. Then $\phi$ can be factored as $\Sigma^\star \to [r]^\star \to \Lambda$ where $[r]$ denotes the set of generators for $\Lambda$ under the above isomorphism. The inverse image of $S$ under $[r]^\star \to \Lambda$ consists of all words where the $r$-tuple of the multiplicities of each $i \in [r]$ taken modulo $n_i$ is one of finitely many values. It is clear that this can be encoded by a DFA with $\#\Lambda$ many states (see \cite[\S 5.2]{catgb}), so is a regular language. Hence $\Sigma^\star_{\phi, S}$ is the inverse image of a regular language in $[r]^\star$ under a monoid homomorphism, and thus is regular \cite[Theorem 3.5]{hopcroftullman}.

(b) Ordered languages are regular by \cite[Theorem 5.3.1]{catgb}. Now use the fact that the intersection of regular languages is regular.
\end{proof}

\begin{definition} \label{defn:pecK}
Let $N \ge 1$ be an integer. We say that $h \in \bQ \lbb t_1, \ldots, t_r \rbb$ is of class $\cK_N$ if it can be written as $f(\bt)/g(\bt)$ where $f$ and $g$ are polynomials in the $t_i$ with coefficients in $\bQ(\zeta_N)$ and $g$ factors as $\prod_{k=1}^n (1-\lambda_i)$, where $\lambda_i$ is a $\bZ[\zeta_N]$-linear combination of the $t_i$. 
\end{definition}

Our main result on quasi-ordered languages is the following theorem. 

\begin{theorem} \label{qordhilb}
Let $\cL$ be a quasi-ordered language of modulus $N$ equipped with a norm valued in $\bN^I$. Then $\rH_{\cL}(\bt)$ is of class $\cK_N$.
\end{theorem}

\begin{proof}
This follows immediately from \cite[Theorem 5.3.7]{catgb} and Proposition~\ref{hilb-cong-intersect}. 
\end{proof}

\begin{remark}
For later applications, it will be convenient to formulate a coordinate-free version of $\cK_N$. Suppose that $\Xi$ is a finite rank free $\bZ$-module and $f \in \SYM(\Xi_{\bQ})$. We say that $f$ is $\cK_N$ if there is a $\bZ$-basis $t_1, \ldots, t_r$ of $\Xi$ so that $f$ is $\cK_N$ as a series in the $t_i$. This is independent of the choice of $\bZ$-basis of $\Xi$.
\end{remark}

Suppose that $\cC$ is directed and normed over $\bN^I$ and pick an object $x$ of $\cC$. We define a norm on $\vert S_x \vert$ as follows: given $f \in \vert S_x \vert$, let $\wt{f} \in S_x(y)$ be a lift, and put $\nu(f)=\nu(y)$. This is well-defined because $\cC$ is directed: if $\wt{f}' \in S_x(z)$ is a second lift, then $y$ and $z$ are isomorphic. 

A {\bf ${\rm QO}_N$-lingual structure} on $\vert S_x \vert$ is a pair $(\Sigma, i)$ consisting of a finite alphabet $\Sigma$ normed over $\bN^I$ and an injection $i \colon \vert S_x \vert \to \Sigma^{\star}$ compatible with the norms, i.e., such that $\nu(i(f))=\nu(f)$ and such that for every poset ideal $J$ of $\vert S_x \vert$, the language $i(J)$ is ${\rm QO}_N$. Following \cite[Definition 6.3.1]{catgb}, a category $\cC$ is {\bf ${\rm QO}_N$-lingual} if $\vert S_x\vert$ admits a ${\rm QO}_N$-lingual structure for all objects $x \in \cC$.

\begin{theorem} \label{hilbthm}
Let $\cC$ be a $\mathrm{QO}_N$-lingual Gr\"obner category and let $M$ be a finitely generated representation of $\cC$. Then $\rH_M(\bt)$ is $\cK_N$, i.e., is a rational function $f(\bt)/g(\bt)$, where $f(\bt)$ and $g(\bt)$ are polynomials with coefficients in $\bQ(\zeta_N)$ and $g(\bt)$ factors as $\prod_{j=1}^n (1-\lambda_j)$ and each $\lambda_j$ is a $\bZ[\zeta_N]$-linear combination of the $t_i$.
\end{theorem}

\begin{proof}
This follows directly from \cite[Theorem~6.3.2]{catgb} and Theorem~\ref{qordhilb}.
\end{proof}

A normed category $\cC$ is {\bf strongly QO${}_N$-lingual} if for each object $x$ there exists a ${\rm QO}_N$-lingual structure $i \colon \vert S_x \vert \to \Sigma^{\star}$ and a congruence language $\cK$ on $\Sigma$ such that for every poset ideal $J$ of $\vert S_x \vert$, the language $i(J)$ is the intersection of an ordered language with $\cK$. (If we drop the adjective ``strongly,'' then the congruence language $\cK$ is allowed to depend on the ideal $J$.) We have the following variant of \cite[Proposition 6.3.3]{catgb}:

\begin{proposition} \label{prop:prod-QON}
Let $\cC_1$ and $\cC_2$ be strongly $\mathrm{QO}_N$-lingual normed categories. Suppose that the posets $\vert \cC_{1,x} \vert$ and $\vert \cC_{2,y} \vert$ are noetherian for all $x$ and $y$. Then $\cC_1 \times \cC_2$ is strongly $\mathrm{QO}_N$-lingual.
\end{proposition}

\begin{proof}
Let $x_j$ be an object of $\cC_j$, and let $i_j \colon \vert \cC_{j,{x_j}} \vert \to \Sigma_j^{\star}$ be a strongly ${\rm QO}_N$-lingual structure at $x_j$. Let $\Sigma=\Sigma_1 \amalg \Sigma_2$, normed over $\bN^I$ in the obvious manner. Let $\cK_1$ and $\cK_2$ be the given congruence languages of modulus $N$ on $\Sigma_1$ and $\Sigma_2$, regarded as languages on $\Sigma$. Write $\cK_i=(\Sigma_i^{\star})_{\phi_i, S_i}$, where $\phi_i \colon \Sigma_i \to \Lambda_i$. Let $\phi \colon \Sigma \to \Lambda_1 \oplus \Lambda_2$ be the map defined by $\phi(x)=(\phi_1(x), 0)$ for $x \in \Sigma_1$, and $\phi(x)=(0, \phi_2(x))$ for $x \in \Sigma_2$. Let $S=S_1 \times S_2$, and let $\cK=\Sigma^{\star}_{\phi, S}$. Then $\cK$ is a congruence language on $\Sigma$ of modulus $N$ and has the following property: if $\cL_1$ and $\cL_2$ are any languages on $\Sigma_1$ and $\Sigma_2$, then $(\cL_1 \cap \cK_1)(\cL_2 \cap \cK_2)=\cL_1 \cL_2 \cap \cK$. 

The following diagram commutes:
\begin{displaymath}
\xymatrix{
\vert \cC_{(x_1,x_2)} \vert \ar@{=}[r] &
\vert \cC_{1,{x_1}} \vert \times \vert \cC_{2,x_2} \vert \ar[r]^-{i_1 \times i_2} \ar[d] & \Sigma_1^{\star} \times \Sigma_2^{\star} \ar[r] \ar[d] & \Sigma^{\star} \ar[d] \\
& \bN^{I_1} \oplus \bN^{I_2} \ar@{=}[r] & \bN^{I_1} \oplus \bN^{I_2} \ar@{=}[r] & \bN^I }
\end{displaymath}
The top right map is concatenation of words. We let $i \colon \vert \cC_{(x_1,x_2)} \vert \to \Sigma^\star$ be the composition along the first line, which is clearly injective. We claim that this is a strongly ${\rm QO}_N$-lingual structure on $\vert \cC_{(x_1,x_2)} \vert$. The commutativity of the above diagram shows that it is a lingual structure. Now suppose $J$ is an ideal of $\vert \cC_{(x_1,x_2)} \vert$. Since this poset is noetherian (being a direct product of noetherian posets), $J$ is a finite union of principal ideals $J_1, \ldots, J_n$. Each $J_j$ is of the form $T_j \times T_j'$, where $T_j$ is an ideal of $\cC_{1,x_1}$ and $T_j'$ is an ideal of $\cC_{2,x_2}$. 

By assumption, $i_1(T_j)$ is the intersection of an ordered language with $\cK_1$, while $i_2(T_j')$ is the intersection of an ordered language with $\cK_2$. It follows from \cite[Theorem 5.3.1]{catgb} that the direct product of two ordered languages on disjoint alphabets is also ordered, so $i(J_j)$ is the intersection of an ordered language with $\cK$. By definition, finite unions of ordered languages are ordered, thus $i(J)$ is the intersection of an ordered language with $\cK$.
\end{proof}

\section{Categories of weighted surjections} \label{sec:weighted-surj}

In this section we study a generalization of the category of finite sets with surjective functions by considering weighted sets. This is preparatory material for the next section.

\subsection{The categories $\OWS_{\Lambda}$ and $\FWS_{\Lambda}$} \label{sec:OWS-FWS}

Let $\Lambda$ be a finite abelian group. A {\bf weighting} on a set $S$ is a function $\phi \colon S \to \Lambda$. A {\bf weighted set} is a set equipped with a weighting; we write $\phi_S$ to denote the weighting if we need to distinguish between sets. Suppose that $\phi$ is a weighting on $S$, and let $f \colon S \to T$ be a function. We define $f_*(\phi)$ to be the weighting on $T$ given by $f_*(\phi)(y) = \sum_{x \in f^{-1}(y)} \phi(x)$. A {\bf map of weighted sets} $S \to T$ is a {\it surjective} function $f \colon S \to T$ such that $f_*(\phi_S)=\phi_T$. In particular, if there is a morphism $S \to T$, then $\#S \ge \#T$. We let $\FWS_{\Lambda}$ denote the category of finite weighted sets.

We require an ordered version of the category as well. Let $\OWS_{\Lambda}$ be the following category. The objects are finite ordered weighted sets. The order and weighting are not required to interact in any way. A morphism $S \to T$ is a surjective function $f \colon S \to T$ such that $f_*(\phi_S)=\phi_T$, and for all $x<y$ in $T$ we have $\min f^{-1}(x) < \min f^{-1}(y)$.

We define a norm $\nu$ on $\OWS_{\Lambda}$ taking values in $\bN^\Lambda$ by defining $\nu(S, \phi_S)$ to be the function $\lambda \mapsto \#\phi_S^{-1}(\lambda)$ (thought of as an element of $\bN^\Lambda$). Our main result about $\OWS_{\Lambda}$ is:

\begin{theorem} \label{thm:ows}
The category $\OWS_{\Lambda}^{\op}$ is Gr\"obner and strongly $\mathrm{QO}_N$-lingual, where $N$ is the exponent of $\Lambda$.
\end{theorem}

We prove this in the next section, and now use it to study $\FWS_{\Lambda}$.

\begin{theorem} \label{thm:FWS-QG}
The category $\FWS_{\Lambda}^{\op}$ is quasi-Gr\"obner.
\end{theorem}

\begin{proof}
The forgetful functor $\Phi \colon \OWS_\Lambda^\op \to \FWS_\Lambda^\op$ is easily seen to satisfy property~(F): given $T \in \OWS$, $S \in \FWS$, and $f \colon \Phi(T) \to S$, there is a unique ordering on $S$ so that $f$ is order-preserving, so $(\Phi^\op)^*(P_S) \cong \bigoplus P_{S'}$ where the sum is over all choices $S'$ of orderings on $S$. So the result follows from Theorem~\ref{thm:ows}.
\end{proof}

\begin{corollary} \label{fws-noeth}
If $\bk$ is left-noetherian then $\Rep_\bk(\FWS_\Lambda^\op)$ is noetherian.
\end{corollary}

We now give a generalization of the above theorem that we will need in our analysis of $\FS_G^{\op}$. Let $M$ be a finitely generated $(\FWS_{\Lambda_1}^{\op} \times \cdots \times \FWS_{\Lambda_r}^{\op})$-module. Enumerate $\Lambda_i$ as $\{\lambda_{i,j}\}$, and let $t_{i,j}$ be a formal variable corresponding to $\lambda_{i,j}$. Given $\bn \in \bN^{\# \Lambda_i}$, let $[\bn]$ be the $\Lambda_i$-weighted set where $n_j$ elements have weight $\lambda_{i,j}$. When $\bk$ is a field, define the Hilbert series of $M$ by
\begin{displaymath}
\rH_M(\bt) = \sum_{\bn(1), \ldots, \bn(r)} C_{\bn(1)} \cdots C_{\bn(r)} \dim_\bk M([\bn(1)], \ldots, [\bn(r)]) \cdot \bt_1^{\bn(1)} \cdots \bt_r^{\bn(r)},
\end{displaymath}
where for $\bn \in \bN^k$ we write $C_{\bn}$ for the multinomial coefficient 
\[
C_\bn = \frac{\vert \bn \vert!}{\bn !} = \frac{\vert \bn \vert!}{n_1! \cdots n_k!}.
\]
The reason for the multinomial coefficient is in the proof of the next theorem, which is the main result we need in our applications.

\begin{theorem} \label{FSWhilb}
Let $M$ be a finitely generated $(\FWS_{\Lambda_1}^{\op} \times \cdots \times \FWS_{\Lambda_r}^{\op})$-module and assume $\bk$ is a field. Then $\rH_M(\bt)$ is $\cK_N$ {\rm (}see Definition~\ref{defn:pecK}{\rm )} where $N$ is the exponent of $\Lambda_1 \times \cdots \times \Lambda_r$.
\end{theorem}

\begin{proof}
Using Theorem~\ref{thm:ows}, $\OWS_{\Lambda_1}^\op \times \cdots \times \OWS_{\Lambda_r}^\op$ is Gr\"obner by Proposition~\ref{grobprod}, and is (strongly) ${\rm QO}_N$-lingual by Proposition~\ref{prop:prod-QON}. Now the result follows from Theorem~\ref{thm:ows} using a version of the functor $\Phi$ from the proof of Theorem~\ref{thm:FWS-QG}. There are $C_{\bn}$ isomorphism classes in $\OWS_{\Lambda_i}$ which map to the isomorphism class of $[\bn]$ in $\FWS_{\Lambda_i}$; thus, with the multinomial coefficients, we have $\rH_M=\rH_{\Phi^*(M)}$.
\end{proof}

\subsection{Proof of Theorem~\ref{thm:ows}}

Fix a finite set $L$ and let $\Sigma=L \times \Lambda$. Given $a \in L$ and $\alpha \in \Lambda$, we write $\frac{a}{\alpha}$ for the corresponding element of $\Sigma$. We denote elements of $\Sigma^\star$ by $\frac{s}{\sigma}$, where $s \in L^{\star}$ and $\sigma \in \Lambda^\star$ are words of equal length. For $a \in L$, we define $w_a \colon \Sigma^{\star} \to \Lambda$ by
\begin{displaymath}
w_a \left( \frac{s_1 \cdots s_n}{\sigma_1 \cdots \sigma_n} \right) = \sum_{s_i=a} \sigma_i.
\end{displaymath}
We let $\bw \colon \Sigma^\star \to \Lambda^L$ be $(w_a)_{a \in L}$. Note that $w_a$ and $\bw$ are monoid homomorphisms. For $\theta \in \Lambda^L$, we let $\cK_{\theta}$ be the set of all $\frac{s}{\sigma} \in \Sigma^\star$ with $\bw(\frac{s}{\sigma})=\theta$. This is a congruence language of modulus $N$ (the exponent of $\Lambda$).

We now define a partial order on $\Sigma^{\star}$. Let $\frac{s}{\sigma} \colon [n] \to \Sigma$ and $\frac{t}{\tau} \colon [m] \to \Sigma$ be two words. Define $\frac{s}{\sigma} \le \frac{t}{\tau}$ if there exists an ordered surjection $f \colon [m] \to [n]$ (i.e., $i < j$ implies $\min f^{-1}(i) < \min f^{-1}(j)$ for all $i,j \in [n]$) such that $t=f^*(s)$ and $\sigma=f_*(\tau)$, i.e., $\sigma(i) = \sum_{j \in f^{-1}(i)} \tau(j)$ for $i \in [n]$; we say $f$ {\bf witnesses} $\frac{s}{\sigma} \le \frac{t}{\tau}$. Note that if $\frac{s}{\sigma} \le \frac{t}{\tau}$ then $\bw(\frac{s}{\sigma})=\bw(\frac{t}{\tau})$.

Let $x=\frac{s}{\sigma}$ be a word in $\Sigma^{\star}$. We say that an index $i$ is \emph{special} for $x$ if $s_j \ne s_i$ for any $j<i$, that is, the letter $s_i$ appears nowhere before $i$ in $s$. 

\begin{lemma} \label{lem:sum-1}
Let $x \le y$ be two words in $\Sigma^{\star}$ of lengths $n$ and $m$. Suppose that the final letters of $x$ and $y$ are equal, say to $\frac{a}{\alpha}$. Suppose furthermore that $n$ is special for $x$ if and only if $m$ is special for $y$. Then we can find an ordered surjection $f \colon [m] \to [n]$ witnessing $x \le y$ and satisfying $f^{-1}(n)=\{m\}$.
\end{lemma}

\begin{proof}
Write $x=\frac{s}{\sigma}$ and $y=\frac{t}{\tau}$, and let $g \colon [m] \to [n]$ be an ordered surjection witnessing $x \le y$, so that $g^*(s)=t$ and $g_*(\tau)=\sigma$. If $a$ appears only once in $x$ and $y$ then $g^{-1}(n)=\{m\}$, so we can take $f=g$. Suppose now that $a$ appears at least twice in both $x$ and $y$. We consider two cases.

First suppose that $g(m)=n$. Write $g^{-1}(m)=S \amalg \{m\}$ for some $S \subset [m]$. We note that $\sum_{i \in S} \tau_i=0$. Let $k<n$ be such that $s_k=a$. Define $f \colon [m] \to [n]$ by $f(i)=g(i)$ for $i \not\in S$ and $f(i)=k$ for $i \in S$. Then $f^{-1}(n)=\{m\}$. It is clear that $f$ is a surjection and satisfies $f^*(s)=t$. We have $\min f^{-1}(i)=\min g^{-1}(i)$ for all $i \ne n$, and so $f$ is ordered since $g$ is. Finally, $f_*(\tau)=\sigma$ since $\tau$ sums to 0 over $S$. Thus $f$ witnesses $x \le y$.

Now suppose that $g(m)=k \ne n$. Write $g^{-1}(k)=S \amalg \{m\}$ and $g^{-1}(n)=T$. Since $\min(S \amalg \{m\})<\min(T)$, it follows that $S$ is non-empty and $\min(S)<\min(T)$. We have $\tau_m+\sum_{i \in S} \tau_i=\sigma_k$ and $\sum_{i \in T} \tau_i=\sigma_n$. Since $\tau_m=\sigma_n$ by assumption, we see that $\sum_{i \in S \cup T} \tau_i=\sigma_k$. Define $f \colon [m] \to [n]$ by $f(i)=g(i)$ for $i \not\in T \cup \{m\}$, $f(i)=k$ for $i \in T$, and $f(m)=n$. Then $f^{-1}(n)=\{m\}$. We have already explained that $f_*(\tau)=\sigma$, and so the reasoning of the previous paragraph shows that $f$ is an ordered surjection witnessing $x \le y$.
\end{proof}

\begin{lemma} \label{lem:sum-2}
Let $x \le y$ be two words in $\Sigma^{\star}$ of lengths $n$ and $m$, and let $0 \le r \le n$. Suppose that:
\begin{enumerate}[\rm \indent (a)]
\item The words formed from the final $r$ letters of $x$ and $y$ are equal.
\item For $i=0,\ldots,r-1$, the index $n-i$ is special for $x$ if and only if $m-i$ is for $y$.
\end{enumerate}
Then we can find an ordered surjection $f \colon [m] \to [n]$ witnessing $x \le y$ and satisfying $f^{-1}(n-i)=\{m-i\}$ for $i=0,\ldots,r-1$.
\end{lemma}

\begin{proof}
We prove the lemma by induction on $r$. The $r=0$ case is trivial. Suppose now that $r>0$ and that the lemma has been proven for all smaller values of $r$. By Lemma~\ref{lem:sum-1}, we can find an ordered surjection $h \colon [m] \to [n]$ witnessing $x \le y$ and satisfying $h^{-1}(n)=\{m\}$. Let $x'$ and $y'$ be the words formed by deleting the final letters of $x$ and $y$. Then $x' \le y'$, as witnessed by $h \vert_{[m-1]}$. By the inductive hypothesis, we can find $f' \colon [m-1] \to [n-1]$ witnessing $x' \le y'$ and satisfying $(f')^{-1}\{n-i\}=m-i$ for $i=1,\ldots,r-1$. We now define $f \colon [m] \to [n]$ by $f(i)=f'(i)$ for $i \in [m-1]$ and $f(m)=n$.
\end{proof}

If $x$ is a word of length $n$ then we write $x_{-k}$ for $x_{n-k+1}$. For instance, $x_{-1}$ is the rightmost letter in $x$.

\begin{lemma} \label{lem:sum-4}
Let $x$ and $y$ be two words in $\Sigma^{\star}$ of lengths $n$ and $m$, let $1 \le r \le n$ be an integer, and let $1 \le \beta_1 < \cdots < \beta_p \le r$. Suppose that conditions {\rm (a)} and {\rm (b)} from Lemma~\ref{lem:sum-2} hold. Suppose furthermore that $x' \le y'$, where $x'$ is obtained from $x$ by deleting indices $-\beta_1, \ldots, -\beta_p$, and $y'$ is formed similarly from $y$. Then $x \le y$.
\end{lemma}

\begin{proof}
We first handle the case $p=1$, where it suffices to treat the case $\beta_1=r$, and we assume this to ease notation. The inequality $x' \le y'$ is witnessed by some ordered surjection $f' \colon [m] \setminus \{m-r+1\} \to [n] \setminus \{n-r+1\}$. By Lemma~\ref{lem:sum-2}, we can assume that $(f')^{-1}\{n-i\}=\{m-i\}$ for $i=0,\ldots,r-2$. We now define $f \colon [m] \to [n]$ by $f(i)=f'(i)$ for $i \ne m-r+1$ and $f(m-r+1)=n-r+1$. This is an ordered surjection witnessing $x \le y$.

Now we do $p$ general by induction on $p$. Let $x''$ be obtained from $x$ be deleting letters $-\beta_1, \ldots, -\beta_{p-1}$, and let $y''$ be obtained from $y$ similarly. Then $x'$ is obtained from $x''$ by deleting index $-\beta_p$ (using the indexing in $x$), and $y'$ is obtained similarly from $y''$. Note that $-\beta_p$ is special in $x''$ if and only if it is in $x$, since the two words only differ to the right of $-\beta_p$. It follows that the specialness of $-\beta_p$ is the same for $x''$ and $y''$. The $p=1$ case shows that $x''\le y''$. By induction on $p$ we now have $x \le y$.
\end{proof}

\begin{lemma} \label{lem:sum-0}
Any $\sigma \in \Lambda^\star$ with $\ell(\sigma) > \#\Lambda$ contains nonempty subsequences that sum to $0$.
\end{lemma}

\begin{proof}
The partial sums $\sigma_1 + \cdots + \sigma_k$ for $k=1,\dots,\ell(\sigma)$ must contain a repeated element, and the difference of these two partial sums is the desired set of elements that sums to $0$.
\end{proof}

\begin{lemma} \label{lem:sum-5}
Let $x=\frac{s}{\sigma}$ be a word in $\Sigma^{\star}$, and put $r=\# L \cdot (\# \Lambda+2)$. Then we can find $1 \le \beta_1 < \cdots < \beta_p < \gamma \le r$ such that $s_{-\beta_1}=\cdots=s_{-\beta_p}=s_{-\gamma}$ and $\sigma_{-\beta_1}+\cdots+\sigma_{-\beta_p}=0$.
\end{lemma}

\begin{proof}
Some letter of $L$, say $a$, must appear at least $k=\# \Lambda+2$ times in the final $r$ letters of $s$. Suppose that $1 \le \alpha_1 < \cdots < \alpha_k \le r$ are such that $s_{-\alpha_i}=a$ for all $1 \le i \le k$. Put $\gamma=\alpha_k$. By Lemma~\ref{lem:sum-0}, we can take $\{\beta_1,\ldots,\beta_p\}$ to be some subset of $\{\alpha_1, \ldots, \alpha_{k-1}\}$.
\end{proof}

\begin{proposition} \label{prop:sigmastar-noeth}
The poset $\Sigma^{\star}$ is noetherian.
\end{proposition}

\begin{proof}
We apply Nash--Williams theory \cite{nash-williams}. Recall that an infinite sequence $x_1,x_2,\ldots$ is \emph{bad} if $x_i \not\le x_j$ whenever $i<j$, and \emph{minimal bad} if it is bad and for each $n$ there is no bad sequence $x_1,\ldots,x_{n-1},y_n,y_{n+1},\ldots$ with $y_n<x_n$. Suppose that $\Sigma^{\star}$ is not noetherian. Then a minimal bad sequence $x_1,x_2,\ldots$ exists: this is a standard consequence of $\Sigma^{\star}$ being well-founded (no infinite decreasing sequences), which is obvious. Let $S \subset \Sigma^*$ be the set of words $y$ such that $y<x_i$ for some $i$. Then $(S, <)$ is noetherian: if not, pick a bad sequence $y_1,y_2,\dots$ in $S$. Then $y_1 < x_j$ for some $j$ and then $x_1,x_2,\dots,x_{j-1}, y_1,y_2, \dots$ is a bad sequence in $\Sigma^\star$ which violates minimality of $x_1,x_2,\dots$. We now proceed to reach a contradiction.

Let $r=\#L \cdot (\# \Lambda+2)$. Let $y_1,y_2,\ldots$ be a subsequence of $x_1,x_2,\ldots$ such that $\ell(y_i) \ge r$ for all $i$, the final $r$ letters of $y_i$ are independent of $i$, and the specialness of the final $r$ indices is independent of $i$. Such a subsequence exists since $\ell(x_i) \to \infty$ and there are only finitely many possibilities for the final $r$ letters and their specialness. Write $y_i=\frac{s_i}{\sigma_i}$. Fix $1 \le \beta_1 < \cdots < \beta_p<\gamma \le r$ such that $s_{i,-\beta_1}=\cdots=s_{i,-\beta_p}=s_{i,-\gamma}$ and $\sigma_{i,-\beta_1}+\cdots+\sigma_{i,-\beta_p}=0$ for all $i$. Such numbers exist by Lemma~\ref{lem:sum-5} and the fact that the final $r$ letters of $y_i$ are independent of $i$.

Let $z_i$ be obtained from $y_i$ by deleting the letters at $-\beta_1, \ldots, -\beta_p$. Then $z_i<y_i$ for all $i$. Indeed, suppose $\ell(y_i)=n$, and put $T=\{n-\beta_1+1,\ldots,n-\beta_p+1\}$. We regard $z_i$ as indexed by $[n] \setminus T$. The ordered surjection $f \colon [n_i] \to [n_i] \setminus T$ defined by $f(i)=i$ for $i \not\in T$ and $f(i)=n-\gamma+1$ for $i \in T$ then witnesses the claimed inequality. The main point here is that $\sum_{j \in T} \sigma_{i,j}=0$, which ensures the appropriate identity on $f_*$.

Since $z_i \in S$ for all $i$ and $S$ is noetherian, there exists $i<j$ such that $z_i<z_j$. But then $y_i<y_j$ by Lemma~\ref{lem:sum-4}, which contradicts the badness of the original sequence $x_1,x_2,\ldots$. It follows that $\Sigma^{\star}$ is noetherian.
\end{proof}

Fix $\theta \in \Lambda^L$. Let $\frac{s}{\sigma}=\frac{s_1\cdots s_n}{\sigma_1 \cdots \sigma_n}$ be a word in $\cK_{\theta}$ (recall this means that $\bw(\frac{s}{\sigma}) = \theta$). Put
\begin{displaymath}
\Pi_i = \left\{ {\textstyle \frac{s_1}{\ast}}, \cdots, {\textstyle \frac{s_i}{\ast}} \right\},
\end{displaymath}
where $\ast$ means any element of $\Lambda$. Define a language $\cL(\frac{s}{\sigma})$ by
\begin{displaymath}
\cL({\textstyle \frac{s}{\sigma}}) = ( {\textstyle \frac{s_1}{\sigma_1}} ) \Pi_1^{\star} \cdots ({\textstyle \frac{s_n}{\sigma_n}}) \Pi_n^{\star}.
\end{displaymath}
It is clear that $\cL(\frac{s}{\sigma})$ is an ordered language.

\begin{lemma} \label{lem:ows-2}
If $\frac{t}{\tau} \in \cL(\frac{s}{\sigma}) \cap \cK_{\theta}$ then $\frac{s}{\sigma} \le \frac{t}{\tau}$.
\end{lemma}

\begin{proof}
Let $\frac{t}{\tau} \in \cL(\frac{s}{\sigma}) \cap \cK_{\theta}$ be a word of length $m$ and write $\frac{t}{\tau}=(\frac{s_1}{\sigma_1}) w_1 \cdots (\frac{s_n}{\sigma_n}) w_n$, with $w_i \in \Pi_i^{\star}$. Let $J \subset [m]$ be the indices occurring in the words $w_1, \ldots, w_n$ and let $K$ be the complement of $J$, so that $\frac{t}{\tau} \vert_K = \frac{s}{\sigma}$. Define a function $f \colon [m] \to [n]$ as follows. On $K$, we let $f$ be the unique order-preserving bijection. For $a \in \{s_1, \ldots, s_n\}$, let $r(a) \in [n]$ be minimal so that $s_{r(a)}=a$. Now define $f$ on $J$ by $f(j)=r(t_j)$. It is clear that $f$ is an ordered surjection and that $f^*(s)=t$. Since $\bw(\frac{s}{\sigma})=\bw(\frac{t}{\tau})=\theta$, it follows that $\bw(\frac{t}{\tau} \vert_J)=0$. From the way we defined $f$, it follows that $f_*(\tau \vert_J)=0$. Thus $f_*(\tau)=\sigma$, which completes the proof.
\end{proof}

We say that a word $\sigma_1 \cdots \sigma_n \in \Lambda^{\star}$ is {\bf minimal} if no non-empty subsequence of $\sigma_2 \cdots \sigma_n$ sums to $0$. Note that we started with the second index. By convention, words of length $1$ are minimal. By Lemma~\ref{lem:sum-0}, there are only finitely many minimal words. Let $\frac{s}{\sigma}$ in $\cK_{\theta}$ be given. We say that $\frac{t}{\tau} \colon [m] \to \Sigma^{\star}$ is {\bf minimal} over $\frac{s}{\sigma} \colon [n] \to \Sigma^{\star}$ if there is an ordered surjection $f \colon [m] \to [n]$ such that $t=f^*(s)$ and $\sigma=f_*(\tau)$ and for every $i \in [n]$ the word $\tau \vert_{f^{-1}(i)}$ is minimal. If $\frac{t}{\tau}$ is minimal over $\frac{s}{\sigma}$ then the length of $\frac{t}{\tau}$ is bounded, so there are only finitely many such minimal words.

\begin{lemma} \label{lem:ows-3}
Let $\frac{s}{\sigma} \le \frac{r}{\rho}$ be words in $\cK_{\theta}$. There exists $\frac{t}{\tau}$ minimal over $\frac{s}{\sigma}$ such that $\frac{r}{\rho} \in \cL(\frac{t}{\tau})$.
\end{lemma}

\begin{proof}
Let $n,m$ be the lengths of $\frac{s}{\sigma}$ and $\frac{r}{\rho}$, and choose a witness $f \colon [m] \to [n]$ to $\frac{s}{\sigma} \le \frac{r}{\rho}$. Let $I \subset [m]$ be the set of elements of the form $\min f^{-1}(i)$ for $i \in [n]$. Let $K \subset [m]$ be minimal subject to $I \subset K$ and $f_*(\rho \vert_K)=\sigma$. Then $\rho \vert_{f^{-1}(i) \cap K}$ is minimal for all $i \in [n]$. Indeed, if it were not then we could discard a subsequence summing to $0$ and make $K$ smaller. We thus see that $\frac{t}{\tau} = \frac{r}{\rho} \vert_K$ is minimal over $\frac{s}{\sigma}$. If $i \in [m] \setminus K$ then there exists $j<i$ in $I$ with $t_i=t_j$, and so $\frac{r}{\rho} \in \cL(\frac{t}{\tau})$.
\end{proof}

\begin{lemma} \label{lem:ows-4}
Every poset ideal of $\cK_{\theta}$ is of the form $\cL \cap \cK_{\theta}$, where $\cL$ is an ordered language on $\Sigma$.
\end{lemma}

\begin{proof}
It suffices to treat the case of a principal ideal. Thus consider the ideal $S$ generated by $\frac{s}{\sigma} \in \cK_{\theta}$. Let $\frac{t_i}{\tau_i}$ for $1 \le i \le n$ be the words minimal over $\frac{s}{\sigma}$, and let $\cL=\bigcup_{i=1}^n \cL(\frac{t_i}{\tau_i})$. Then $\cL$ is an ordered language, by construction. If $\frac{r}{\rho} \in \cL \cap \cK_{\theta}$ then $\frac{r}{\rho} \in \cL(\frac{t_i}{\tau_i}) \cap \cK_{\theta}$ for some $i$, and so $\frac{s}{\sigma} \le \frac{t_i}{\tau_i} \le \frac{r}{\rho}$ by Lemma~\ref{lem:ows-2}, and so $\frac{r}{\rho} \in S$. Conversely, suppose $\frac{r}{\rho} \in S$. Then $\frac{r}{\rho} \in \cL(\frac{t_i}{\tau_i})$ for some $i$ by Lemma~\ref{lem:ows-3}, and of course $\frac{r}{\rho} \in \cK_{\theta}$, and so $\frac{r}{\rho} \in \cL \cap \cK_{\theta}$.
\end{proof}

\begin{proof}[Proof of Theorem~\ref{thm:ows}]
The category $\cC=\OWS_{\Lambda}^{\op}$ is clearly directed. Let $x=([n], \theta)$ be an object of $\cC$. To show that $\cC$ is Gr\"obner, we need to show that $\vert \cC_x \vert$ is orderable and is a noetherian poset. Lexicographic order on $\Sigma^{\star}$ induces an admissible order on $\vert \cC_x \vert$. 

To show that $\vert \cC_x \vert$ is noetherian, we apply the above theory with $L=[n]$. Suppose that $f \colon x \to y$ is a map in $\cC$, with $y=([m], \phi)$; note that this means that $f$ is a surjection $[m] \to [n]$. We define a word $[m] \to \Sigma^{\star}$ by mapping $j \in [m]$ to $(f(j), \phi(j))$. One can reconstruct $f$ from this word, so this defines an injection $i \colon \vert \cC_x \vert \to \Sigma^{\star}$. In fact, the image lands in $\cK_{\theta}$. It is clear from the definition of the order on $\Sigma^{\star}$ that $i$ is strictly order-preserving. Thus $\vert \cC_x \vert$ is noetherian by Proposition~\ref{prop:sigmastar-noeth}. 

Finally, since $i$ maps ideals to ideals, we see that it gives a strong QO${}_N$-lingual structure on $\vert \cC_x \vert$ by Lemma~\ref{lem:ows-4}.
\end{proof}

\section{Categories of $G$-surjections} \label{sec:FSG}

In this section we study the surjective version of $G$-sets. The noetherian property is deduced in \S\ref{ss:FSG-basic}. In \S \ref{ss:gendelta}, we explain the connection between $\FS_G^{\op}$-modules and $\Delta$-modules. In \S \ref{ss:fsexample}, we give some examples of $\FS_G^{\op}$-modules. The results on Hilbert series are stated in \S\ref{sec:Gsurj-results}. The remainder of the section is devoted to proving the Hilbert series results.

Throughout, we assume that $G$ is finite. Following \cite[\S 14]{serre}, we let $\cR_\bk(G)$ denote the Grothendieck group of finitely generated $\bk[G]$-modules.

\subsection{Basic properties} \label{ss:FSG-basic}

Let $\Phi \colon \FS \to \FS_G$ be the functor taking a function $f \colon S \to T$ to the $G$-function $(f,\sigma) \colon S \to T$ where $\sigma=1$.

\begin{proposition} \label{FSGpropG}
The functor $\Phi^\op \colon \FS^{\op} \to \FS^{\op}_G$ satisfies property~{\rm (F)}.
\end{proposition}

\begin{proof}
Let $x \in \FS_G$ be given. Say that a morphism $(f, \sigma) \colon y \to x$ is minimal if the function $(f, \sigma) \colon y \to x \times G$ is injective. Since $G$ is finite, minimal implies $\#y \le \#x \cdot \#G$, so there are finitely many minimal maps up to isomorphism. Now consider a map $(f, \sigma) \colon y \to x$ in $\FS_G$. Define an equivalence relation on $y$ by $a \sim b$ if $f(a)=f(b)$ and $\sigma(a)=\sigma(b)$, and let $g \colon y \to y'$ be the quotient. Then the induced map $(f', \sigma) \colon y' \to x$ is minimal. Furthermore, $(f, \sigma)=(g,1)(f', \sigma)=\Phi(g) (f', \sigma)$. Reversing all of the morphisms, we see that $\Phi^{\op}$ satisfies property~(F).
\end{proof}

Given a finite collection $\ul{G}=(G_i)_{i \in I}$ of finite groups, write $\FS_{\ul{G}} = \prod_{i \in I} \FS_{G_i}$.

\begin{corollary}
The category $\FS_{\ul{G}}^{\op}$ is quasi-Gr\"obner.
\end{corollary}

\begin{proof}
This follows from Propositions~\ref{potgrob}, \ref{grobprod}, and \cite[Theorem 8.1.2]{catgb}.
\end{proof}

\begin{corollary} \label{FSG-noeth}
If $\bk$ is left-noetherian then $\Rep_{\bk}(\FS^{\op}_{\ul{G}})$ is noetherian.
\end{corollary}

\subsection{Generalized $\Delta$-modules}
\label{ss:gendelta}

Let $\cA$ be an abelian category equipped with a symmetric ``cotensor'' structure, i.e., a functor $\cA \to \cA \otimes \cA$, and analogous data opposite to that of a tensor structure. (Here we are using the Deligne tensor product of abelian categories \cite{deligne}.) Given a surjection $f \colon T \to S$ of finite sets, there is an induced functor $f^* \colon \cA^{\otimes S} \to \cA^{\otimes T}$ by cotensoring along the fibers of $f$. A {\bf $\Delta$-module} over $\cA$ is a rule $M$ that assigns to each finite set $S$ an object $M_S$ of $\cA^{\otimes S}$ and to each surjection $f \colon T \to S$ of finite sets a morphism $M_f \colon f^*(M_S) \to M_T$, such that if $f \colon T \to S$ and $g \colon S \to R$ are surjections, then the diagram
\begin{displaymath}
\xymatrix{
& g^*(M_S) \ar[dr]^{M_g} \\
(gf)^*(M_R) \ar[ur]^{f^*(M_g)} \ar[rr]^{M_{gf}} && M_T }
\end{displaymath}
commutes. There are two main examples relevant to this paper:
\begin{itemize}
\item Let $\cA$ be the category of polynomial functors $\Vec \to \Vec$. Then $\cA^{\otimes 2}$ is identified with the category of polynomial functors $\Vec^2 \to \Vec$. There is a comultiplication $\cA \to \cA^{\otimes 2}$ taking a functor $F$ to the functor $(U, V) \mapsto F(U \otimes V)$, and this gives $\cA$ the structure of a symmetric cotensor category. $\Delta$-modules over $\cA$ are $\Delta$-modules as defined in \cite[\S 9.2]{catgb}.
\item Let $\cA$ be the category of representations of a finite group $G$. Then $\cA^{\otimes 2}$ is identified with the category of representations of $G \times G$. There is a comultiplication $\cA \to \cA^{\otimes 2}$ taking a representation $V$ of $G$ to the representation $\Ind_G^{G \times G}(V)$ of $G \times G$, where $G$ is included in $G \times G$ via the diagonal map. This gives $\cA$ the structure of a symmetric cotensor category. $\Delta$-modules over $\cA$ are representations of $\FS_G^{\op}$.
\end{itemize}

If $n!$ is invertible in the base field then the category of polynomial functors of degree $\le n$ is equivalent, as a cotensor category, to the category $\prod_{k=0}^n \Rep(S_k)$.
To see this, first note that invertibility of $n!$ implies that $\Rep(S_k)$ for $k \le n$, as well as the category of polynomial functors of degree $k \le n$ are semisimple categories, and are equivalent by Schur--Weyl duality (this is usually stated over characteristic $0$ \cite[\S 4]{expos}, but all that is needed is $n!$ to be invertible). The compatibility of the cotensor structures follows from the form of the Schur--Weyl equivalence.

We thus find that $\Delta$-modules of degree $\le n$ (in the sense of \cite[\S 9.1]{catgb}) coincide with representations of $\prod_{k=0}^n \FS_{S_k}^{\op}$. Thus our results on $\FS_G^{\op}$ can be loosely viewed as a generalization of our results on $\Delta$-modules from \cite{catgb} (``loosely'' because in bad characteristic the results are independent of each other). It seems possible that our results could generalize to $\Delta$-modules over any ``finite'' abelian cotensor category.

\subsection{Examples: Segre products of simplicial complexes}
\label{ss:fsexample}

We now give a source of $\FS_G^{\op}$-modules. Let $X$ and $Y$ be simplicial complexes on finite vertex sets $X_0$ and $Y_0$. Define a simplicial complex $X \ast Y$ on the vertex set $X_0 \times Y_0$ as follows. Let $p_1 \colon X_0 \times Y_0 \to X_0$ be the projection map, and similarly define $p_2$. Then $S \subset X_0 \times Y_0$ is a simplex if and only if $p_1(S)$ and $p_2(S)$ are simplices of $X$ and $Y$ and have the same cardinality as $S$. We call $X \ast Y$ the {\bf Segre product} of $X$ and $Y$. It is functorial for maps of simplicial complexes. It is not a topological construction, and depends in an essential way on the simplicial structure.

Fix a finite simplicial complex $X$, equipped with an action of a group $G$. The diagonal map $X_0 \to X_0 \times X_0$ induces a map of simplicial complexes $X \to X \ast X$. We thus obtain a functor from $\FS_G^{\op}$ to the category of simplicial complexes by $S \mapsto X^{\ast S}$. Fixing $i$, we obtain a representation $M_i$ of $\FS_G^{\op}$ by $S \mapsto \rH_i(X^{\ast S}; \bk)$. It is not difficult to directly show that $S \mapsto \rC_i(X^{\ast S}; \bk)$ is a finitely generated representation of $\FS_G^{\op}$, where $\rC_i$ denotes the space of simplicial $i$-chains. Thus by Corollary~\ref{FSG-noeth}, $M_i$ is a finitely generated representation of $\FS_G^{\op}$. Theorem~\ref{FSGhilb} below gives information about the Hilbert series of $M_i$.

The case where $X$ is just a single simplex is already extremely complicated and interesting, and is closely related to syzygies of the Segre embedding. In fact, if $X$ has $d$ vertices then the $\FS_{S_d}^{\op}$-module given by $\rH_{p-1}(X^{\ast \bullet}; \bk)$ coincides with the degree $d$ piece of the $\Delta$-module $F_p$ of $p$-syzygies of the Segre embedding (as defined in \cite{delta-mod}) under the equivalence in the previous section, at least when $d!$ is invertible in $\bk$.

\subsection{Hilbert series} \label{sec:Gsurj-results}

Let $M$ be a finitely generated $\FS_{\ul{G}}^{\op}$-module over an algebraically closed field $\bk$. We need the following fact: given finite groups $G$ and $H$, every irreducible $\bk[G \times H]$-module is of the form $V \otimes W$ where $V$ and $W$ are irreducible modules for $\bk[G]$, respectively $\bk[H]$ \cite[Proposition 2.3.23]{kowalski}. This implies that we have a canonical identification 
\[
\cR_\bk(G \times H) = \cR_\bk(G) \otimes \cR_\bk(H).
\]

Let $\bn \in \bN^I$, and write $[\bn]$ for $([n_i])_{i \in I}$. Then $M({[\bn]})$ is a finite-dimensional representation of $\ul{G}^{\bn}$. Let $[M]_{\bn}$ denote the image of the class of this representation under the map 
\begin{displaymath}
\cR_{\bk}(\ul{G}^{\bn}) = \bigotimes_{i \in I} \cR_{\bk}(G_i)^{\otimes n_i} \to \Sym^{\vert \bn \vert}(\cR_{\bk}(\ul{G})),
\end{displaymath}
where
\begin{displaymath}
\cR_{\bk}(\ul{G}) = \bigoplus_{i \in I} \cR_{\bk}(G_i).
\end{displaymath}
Note that in good characteristic, one can recover the isomorphism class of $M([\bn])$ as a representation of $\ul{G}^{\bn}$ from $[M]_{\bn}$ due to the $S_{\bn}$-equivariance. If $\{L_{i,j}\}$ are the irreducible representations of the $G_i$, then $[M]_{\bn}$ can be thought of as a polynomial in corresponding variables $\{t_{i,j}\}$. Define the {\bf Hilbert series} of $M$ by 
\begin{displaymath}
\rH_M(\bt) = \sum_{\bn \in \bN^I} [M]_{\bn}.
\end{displaymath}
This is an element of the ring $\SYM(\cR_{\bk}(\ul{G})_{\bQ}) \cong \bQ \lbb t_{i,j} \rbb$. This definition does not fit into our framework of Hilbert series of normed categories, though it can be seen as an enhanced Hilbert series (similar to \cite[\S 5.1]{symc1}).

The following is a simplified version of our main theorem on Hilbert series. Recall the definition of $\cK_N$ from Definition~\ref{defn:pecK}.

\begin{theorem} \label{FSGhilb}
Let $M$ be a finitely generated representation of $\FS_{\ul{G}}^{\op}$ over an algebraically closed field $\bk$. Then $\rH_M(\bt)$ is a $\cK_N$ function of the $\bt$, where $N$ is the least common multiple of the exponents of the $G_i$ not divisible by $p$.
\end{theorem}

Stating the full result requires some additional notions. Let $G$ be a finite group and let $\bk$ be a field. Let $\{H_j\}_{j \in J}$ be a collection of subgroups such that the orders of the commutator subgroups $[H_j,H_j]$ are invertible in $\bk$, and let $H_j^{\ab} = H_j / [H_j,H_j]$ be the abelianization of $H_j$. There is a functor $\Rep_{\bk}(H_j) \to \Rep_{\bk}(H_j^{\ab})$ given by taking coinvariants under $[H_j, H_j]$. This functor is exact since the order of $[H_j,H_j]$ is invertible in $\bk$, and thus induces a homomorphism of Grothendieck groups $\cR_{\bk}(H_j) \to \cR_{\bk}(H_j^{\ab})$. There are also homomorphisms $\cR_{\bk}(G) \to \cR_{\bk}(H_j)$ given by restriction. We say that the family $\{H_j\}$ is {\bf good} if the composite
\begin{displaymath}
\cR_{\bk}(G) \to \bigoplus_{j \in J} \cR_{\bk}(H_j) \to \bigoplus_{j \in J} \cR_{\bk}(H_j^{\ab})
\end{displaymath}
is a split injection (i.e., an injection with torsion-free cokernel). We say that $G$ is {\bf $N$-good} if it admits a good family $\{H_j\}$ such that the exponent of each $H_j^{\ab}$ divides $N$. We say that a family $\ul{G}$ of finite groups is $N$-good if each member is. 

The following is our main theorem on Hilbert series. The proof is given in \S\ref{ss:FSGproof}.

\begin{theorem} \label{FSGhilb2}
Suppose that $\ul{G}$ is $N$-good. Let $M$ be a finitely generated $\FS_{\ul{G}}^{\op}$-module over an algebraically closed field $\bk$. Then $\rH_M(\bt)$ is a $\cK_N$ function of the $t_{i,j}$.
\end{theorem}

Using Brauer's theorem, we show that over an algebraically closed field, every group is $N$-good where $N$ is prime-to-$p$ part of the exponent of the group (Proposition~\ref{FSGlem2}), and so Theorem~\ref{FSGhilb} follows from Theorem~\ref{FSGhilb2}. We show that symmetric groups are 2-good if $n!$ is invertible in $\bk$, which recovers some of our results on Hilbert series of $\Delta$-modules (see \cite[\S 9.2]{catgb}) in good characteristic. For general groups, we know little about the optimal value of $N$. Finding some results could be an interesting group theory problem.

\begin{example} \label{hilbexample}
Let $G$ be a finite group and let $\{V_i\}_{i \in I}$ be the set of irreducible representations of $G$ over $\bC$. Define an $\FS_G^{\op}$-module $M_i$ by
\begin{displaymath}
M_i(S) = \Ind_G^{G^S}(V_i),
\end{displaymath}
where $G \to G^S$ is the diagonal map. Let $C$ be the set of conjugacy classes in $G$, $\chi_i$ be the character of $V_i$, and $t_i$ be an indeterminate corresponding to $V_i$. A computation similar to that in \cite[Lem.~5.7]{delta-mod} gives
\begin{displaymath}
\rH_{M_i}(\bt) = \frac{1}{\# G} \sum_{c \in C} \frac{\# c \cdot \chi_i(c)}{1-(\sum_{j \in I} \chi_j(c) t_j)}.
\end{displaymath}
This is a $\cK_N$ function of the $t_i$, as predicted by Theorem~\ref{FSGhilb}, where $N$ is such that all characters of $G$ take values in $\bQ(\zeta_N)$.
\end{example}

\subsection{Group theory}
\label{ss:gpthy}

Let $p=\chr(\bk)$. If $p=0$ then, by convention, every group has order prime to $p$, and the only $p$-group is the trivial group. We say that a collection $\{H_i\}_{i \in I}$ of subgroups of $G$ is a {\bf covering} if the map on Grothendieck groups
\begin{displaymath}
\cR_{\bk}(G) \to \bigoplus_{i \in I} \cR_{\bk}(H_i)
\end{displaymath}
is a split injection. Recall that if $\ell$ is a prime, then an {\bf $\ell$-elementary group} is one that is the direct product of an $\ell$-group and a cyclic group of order prime to $\ell$. An {\bf elementary group} is a group which is $\ell$-elementary for some prime $\ell$. 

\begin{lemma} \label{lem:good-groups}
The following results hold over any field $\bk$:
\begin{enumerate}[\indent \rm (a)]
\item Let $\{H_i\}_{i \in I}$ be a covering of $G$, and for each $i$ let $\{K_j\}_{j \in J_i}$ be a covering of $H_i$, and let $J=\amalg_{i \in I} J_i$. Then $\{K_j\}_{j \in J}$ is a covering of $G$.

\item Let $\{H_i\}_{i \in I}$ be a covering of $G$, and suppose each $H_i$ is $N$-good. Then $G$ is $N$-good.

\item Suppose that $H$ is a $p$-elementary group and write $H=H_1 \times H_2$, where $H_1$ is cyclic of order prime to $p$ and $H_2$ is a $p$-group. Then $\{H_1\}$ is a covering of $H$.
\end{enumerate}
The following hold if $\bk$ is algebraically closed:
\begin{enumerate}[\indent \rm (a)]
\setcounter{enumi}{3}
\item The collection of elementary subgroups $\{H_i\}_{i \in I}$ of $G$ is a covering of $G$.

\item If $H$ has order prime to $p$, then $H$ is $N$-good where $N$ is the exponent of $H$.
\end{enumerate}
\end{lemma}

\begin{proof}
(a) and (b) are clear.

(c) The only simple $\bk[H_2]$-module is trivial, so $\cR_{\bk}(H) \to \cR_{\bk}(H_1)$ is an isomorphism. 

(d) Let $\alpha \colon \cR_{\bk}(G) \to \bigoplus \cR_{\bk}(H_i)$ be the restriction map. Let $\cP_{\bk}(G)$ be the Grothendieck group of finite-dimensional projective $\bk[G]$-modules. The map $\cR_{\bk}(G) \times \cP_{\bk}(G) \to \bZ$ given by $(V,W) \mapsto \dim_\bk \Hom_G(V,W)$ is a perfect pairing \cite[\S 14.5]{serre}. Combining this with Frobenius reciprocity, it follows that the dual of $\alpha$ can be identified with the induction map $\bigoplus \cP_{\bk}(H_i) \to \cP_{\bk}(G)$. This map is surjective by Brauer's theorem \cite[\S 17.2, Thm.~39]{serre}. Since the dual of $\alpha$ is surjective, it follows that $\alpha$ is a split injection, which proves the claim.

(e) Indeed, arguing with duals and Frobenius reciprocity again, it is enough to find subgroups $\{K_i\}_{i \in I}$ of $H$ such that the induction map $\bigoplus_{i \in I} \cR_{\bk}(K_i^{\ab}) \to \cR_{\bk}(H)$ is surjective. (Note that $\cR_{\bk}=\cP_{\bk}$ for groups of order prime to $p$.) This follows from Brauer's theorem \cite[\S 10.5, Thm.~20]{serre}. 
\end{proof}

\begin{proposition} \label{FSGlem2}
Suppose that $\bk$ is algebraically closed and $G$ is a finite group. Then $G$ is $N$-good where $N$ is the prime-to-$p$ part of the exponent of $G$.
\end{proposition}

\begin{proof}
By parts (a), (c), and (d) of Lemma~\ref{lem:good-groups}, $G$ has a covering by its subgroups of order prime to $p$. For each of these groups, its set of subgroups is good by part (e). Now finish by applying (b).
\end{proof}

We now construct a good collection of subgroups for the symmetric group $S_n$ in good characteristic. Given a partition $\lambda = (\lambda_1, \dots, \lambda_n)$ with $\sum_i \lambda_i = n$, let $S_\lambda = S_{\lambda_1} \times \cdots \times S_{\lambda_n}$ be the corresponding Young subgroup of $S_n$.

\begin{proposition} \label{Sngood}
Suppose that $n!$ is invertible in $\bk$. Then $\{S_\lambda\}$ is a good collection of subgroups of the symmetric group $S_n$. In particular, $S_n$ is $2$-good.
\end{proposition}

\begin{proof}
Under the assumption on $\chr(\bk)$, the representations of $S_n$ are semisimple. Using Frobenius reciprocity, the restriction map on representation rings is dual to induction. We claim that each irreducible character of $S_n$ is a $\bZ$-linear combination of the permutation representations on $S_n/S_\lambda$ (this implies $\{S_\lambda\}$ is a good collection of subgroups). Recall that the irreducible representations of $S_n$ are indexed by partitions of $n$ (we denote them $\bM_\lambda$). Also, recall the dominance order on partitions: $\lambda \ge \mu$ if $\lambda_1 + \cdots + \lambda_i \ge \mu_1 + \cdots + \mu_i$ for all $i$. An immediate consequence of Pieri's rule \cite[(2.10)]{expos} is that the permutation representation $S_n/S_\lambda$ contains $\bM_\lambda$ with multiplicity $1$ and the remaining representations $\bM_\mu$ that appear satisfy $\mu \ge \lambda$. This proves the claim. 

For the last statement, note that the group $S_{\lambda}^{\ab}$ has exponent $1$ or $2$ for any $\lambda$.
\end{proof}

\subsection{Proof of Theorem~\ref{FSGhilb2}}
\label{ss:FSGproof}

First, we use the good family of subgroups to reduce to the case where each $G_i$ is abelian of invertible order. For such $G$, we identify $\FS^\op_G$-modules with $\FWS^\op_{\Lambda}$-modules, where $\Lambda$ is the group of characters of $G$. The theorem then follows from our results for Hilbert series of $\FWS^\op_{\Lambda}$-modules. We now go through the details.

\begin{lemma} \label{FSGlem3}
Let $\ul{G}=(G_i)_{i \in I}$ be finite groups. For each $i$, let $H_i$ be a subgroup of $G_i$ such that the order of $[H_i,H_i]$ is invertible in $\bk$. Define a functor
\begin{displaymath}
\Phi \colon \Rep_{\bk}(\FS_{\ul{G}}^{\op}) \to \Rep_{\bk}(\FS_{\ul{H}^{\ab}}^{\op})
\end{displaymath}
by letting $\Phi(M)(\ul{S})$ be the $[\ul{H},\ul{H}]^{\ul{S}}$-coinvariants of $M(\ul{S})$. Then we have the following:
\begin{enumerate}[\indent \rm (a)]
\item $\Phi(M)$ is a well-defined object of $\Rep_{\bk}(\FS^\op_{\ul{H}^{\ab}})$.
\item If $M$ is finitely generated then so is $\Phi(M)$.
\item Let $\phi_i \colon \cR_{\bk}(G_i) \to \cR_{\bk}(H_i^{\ab})$ be the map induced by restricting to $H_i$ followed by taking $[H_i,H_i]$-coinvariants, and let $\phi \colon \cR_{\bk}(\ul{G}) \to \cR_{\bk}(\ul{H}^{\ab})$ be the sum of the $\phi_i$. Then $\rH_{\Phi(M)}$ is the image of $\rH_M$ under the ring homomorphism
$\SYM(\cR_{\bk}(\ul{G})_{\bQ}) \to \SYM(\cR_{\bk}(\ul{H}^{\ab})_{\bQ})$ induced by $\phi$.
\end{enumerate}
\end{lemma}

\begin{proof}
(a) For a tuple $\ul{S}=(S_i)_{i \in I}$ of sets, let $K(\ul{S})$ be the $\bk$-subspace of $M(\ul{S})$ spanned by elements of the form $gm-m$ with $g \in [\ul{H},\ul{H}]^{\ul{S}}$ and $m \in M(\ul{S})$. If $f \colon \ul{S} \to \ul{T}$ is a morphism in $\FS^I$ then the induced map $f^* \colon M(\ul{T}) \to M(\ul{S})$ carries $gm-m$ to $f^*(g) f^*(m)-f^*(m)$. Thus $K$ is a $(\FS^{\op})^I$-submodule of $M$, and so $M/K$ is a well-defined $(\FS^{\op})^I$-module. The group actions clearly carry through, and so $\Phi(M)$ is well-defined.

(b) Suppose that $M \in \Rep_{\bk}(\FS^{\op}_{\ul{G}})$ is finitely generated. Then the restriction of $M$ to $(\FS^{\op})^I$ is finitely generated by Propositions~\ref{FSGpropG} and~\ref{propFfg}. The restriction of $\Phi(M)$ to $(\FS^{\op})^I$ is a quotient of the restriction of $M$, and is therefore finitely generated. Thus $\Phi(M)$ is finitely generated, by Proposition~\ref{pullback-fg}.

(c) This is clear.
\end{proof}

\begin{proposition} \label{FSunionG}
The functor $\Phi \colon \FS^{\op} \times \FS^{\op} \to \FS^{\op}$ given by disjoint union satisfies property~{\rm (F)}.
\end{proposition}

\begin{proof}
Pick a finite set $S$. Let $f \colon T_1 \amalg T_2 \to S$ be a surjection. Then this can be factored as $T_1 \amalg T_2 \to f(T_1) \amalg f(T_2) \to S$ where the first map is the image of a morphism $(T_1, T_2) \to (f(T_1), f(T_2))$ under $\Phi^\op$. So for the $y_1, y_2, \dots$ in the definition of property~(F), we take the pairs $(T, T')$ of subsets of $S$ whose union is all of $S$.
\end{proof}

\begin{lemma} \label{FSGlem4}
Let $\ul{G}=(G_i)_{i \in I}$ be groups, let $f \colon J \to I$ be a surjection, and let $f^*(\ul{G})$ be the resulting family of groups indexed by $J$. Let $\Phi \colon \FS_{f^*(\ul{G})}^{\op} \to \FS_{\ul{G}}^{\op}$ be the functor induced by disjoint union, i.e., $\Phi(\{S_j\}_{j \in J})=\{T_i\}_{i \in I}$ where $T_i=\coprod_{j \in f^{-1}(i)} S_j$.
\begin{enumerate}[\indent \rm (a)]
\item $\Phi$ satisfies property~{\rm (F)}; in particular, if $M$ is finitely generated then so is $\Phi^*(M)$.
\item Let $\phi_i \colon \cR_{\bk}(G_i) \to \bigoplus_{j \in f^{-1}(i)} \cR_{\bk}(G_j)$ be the diagonal map, and let $\phi \colon \cR_{\bk}(\ul{G}) \to \cR_{\bk}(f^*(\ul{G}))$ be the sum of the $\phi_i$. Then $\rH_{\Phi^*(M)}$ is the image of $\rH_M$ under the ring homomorphism $\SYM(\cR_{\bk}(\ul{G})_{\bQ}) \to \SYM(\cR_{\bk}(f^*(\ul{G}))_{\bQ})$ induced by $\phi$.
\end{enumerate}
\end{lemma}

\begin{proof}
Consider the commutative diagram of categories
\begin{displaymath}
\xymatrix{
\FS^{\op}_{f^*(\ul{G})} \ar[r]^{\Phi} & \FS^{\op}_{\ul{G}} \\
(\FS^{\op})^J \ar[r]^{\Phi'} \ar[u] & (\FS^{\op})^I \ar[u] }.
\end{displaymath}
The functor $\Phi'$ is defined just like $\Phi$; it satisfies property~(F) by Proposition~\ref{FSunionG}. The vertical maps satisfy property~(F) by Proposition~\ref{FSGpropG}. Thus $\Phi$ satisfies property~(F) by Proposition~\ref{propFcomp}. This proves (a); (b) is clear.
\end{proof}

\begin{lemma} \label{FSGlem5}
Suppose that $\ul{G}=(G_i)_{i \in I}$ is a family of commutative groups of exponents dividing $N$. Suppose that $N$ is invertible in $\bk$ and that $\bk$ contains the $N$th roots of unity. Let $\Lambda_i=\Hom(G_i, \bk^{\times})$ be the group of characters of $G_i$, and let $\ul{\Lambda}=(\Lambda_i)_{i \in I}$. Then there is an equivalence $\Phi \colon \Rep_{\bk}(\FS_{\ul{G}}^{\op}) \to \Rep_{\bk}(\FWS_{\ul{\Lambda}}^{\op})$ respecting Hilbert series, i.e., $\rH_M=\rH_{\Phi(M)}$ for $M \in \Rep_{\bf}(\FS_G^{\op})$.
\end{lemma}

Before giving the proof, we offer two clarifications. First, $\FWS_{\ul{\Lambda}}$ denotes the category $\prod_{i \in I} \FWS_{\Lambda_i}$. An object of this category is a tuple of sets $\ul{S}=(S_i)_{i \in I}$ equipped with a weight function $\phi_i \colon S_i \to \Lambda_i$ for each $i$. Second, $\rH_M$ and $\rH_{\Phi(M)}$ are both series in variables indexed by the characters of the $G_i$. This is why they are comparable.

\begin{proof}
Let $M$ be a representation of $\FS_{\ul{G}}^{\op}$. Let $\ul{S}=(S_i)_{i \in I}$ be a tuple of sets. Then we have a decomposition
\begin{displaymath}
M_{\ul{S}} = \bigoplus M_{\ul{S}, \ul{\phi}},
\end{displaymath}
where the sum is over weightings $\ul{\phi}$ of $\ul{S}$, and $M_{\ul{S}, \ul{\phi}}$ is the subspace of $M_{\ul{S}}$ on which $\ul{G}^{\ul{S}}$ acts through $\ul{\phi}$. If $\ul{f} \colon \ul{S} \to \ul{T}$ is a morphism in $\FS_{\ul{G}}^{\op}$ then the map $\ul{f}{}_* \colon M_{\ul{S}} \to M_{\ul{T}}$ carries $M_{\ul{S}, \ul{\phi}}$ into $M_{\ul{T}, \ul{f}{}_*(\ul{\phi})}$. We define $\Phi(M)$ to be the functor on $\FWS_{\ul{\Lambda}}^{\op}$ which assigns to a weighted set $(\ul{S}, \ul{\phi})$ the space $M_{\ul{S}, \ul{\phi}}$. This construction can be reversed: given a representation $M$ of $\FWS_{\ul{\Lambda}}^{\op}$, we can build a representation of $\FS_{\ul{G}}^{\op}$ by defining $M_{\ul{S}}$ to be the sum of the $M_{\ul{S}, \ul{\phi}}$. We leave to the reader the verification that these constructions are quasi-inverse to each other. This shows that $\Phi$ is an equivalence. It is clear that it preserves Hilbert series: we note that the multinomial coefficients in the definition of $\rH_{\Phi(M)}$ count, for each $M_{\ul{S}, \ul{\phi}}$, the number of $M_{\ul{S}, \ul{\phi'}}$ where $\ul{\phi'}$ is a permutation of $\ul{\phi}$.
\end{proof}

\begin{lemma} \label{FSGlem1}
Let $i \colon \Xi \to \Xi'$ be a split injection of finite rank free $\bZ$-modules, and pick $f \in \SYM(\Xi_{\bQ})$. Suppose that $i(f) \in \SYM(\Xi'_{\bQ})$ is $\cK_N$. Then $f$ is $\cK_N$.
\end{lemma}

\begin{proof}
Let $j \colon \Xi' \to \Xi$ be a splitting of $i$. Then $f=j(i(f))$. Since $j$ clearly takes $\cK_N$ functions to $\cK_N$ functions, it follows that $f$ is $\cK_N$.
\end{proof}

\begin{proof}[Proof of Theorem~\ref{FSGhilb2}]
Let $M$ be a finitely generated representation of $\FS^{\op}_{\ul{G}}$, where $\ul{G}=(G_i)_{i \in I}$. For each $i \in I$, let $\{H_j\}_{j \in J_i}$ be a good collection of subgroups of $\ul{G}$ such that the exponent of each $H_j^{\ab}$ divides $N$. Let $J=\coprod_{i \in I} J_i$ and let $f \colon J \to I$ be the projection map. Then we have functors
\begin{displaymath}
\Rep_{\bk}(\FS^{\op}_{\ul{G}}) \to \Rep_{\bk}(\FS^{\op}_{f^*(\ul{G})}) \to \Rep_{\bk}(\FS^{\op}_{\ul{H}^{\ab}}).
\end{displaymath}
Let $M'$ be the image of $M$ under the composition. By Lemmas~\ref{FSGlem3} and~\ref{FSGlem4}, $M'$ is finitely generated and $\rH_{M'}$ is the image of $\rH_M$ under the ring homomorphism corresponding to the natural additive map $\cR_{\bk}(\ul{G}) \to \cR_{\bk}(\ul{H}^{\ab})$. By Lemma~\ref{FSGlem5} and Theorem~\ref{FSWhilb}, $\rH_{M'}$ is $\cK_N$. Thus by Lemma~\ref{FSGlem1}, $\rH_M$ is $\cK_N$. This completes the proof.
\end{proof}


\begin{thebibliography}{CEFN}

\bibitem[B]{betley} Stanislaw Betley, Twisted homology of symmetric groups, {\it Proc. Amer. Math. Soc.} {\bf 130} (2002), no.~12, 3439--3445.

\bibitem[Ca]{casto} Kevin Casto, ${\rm FI}_G$-modules, orbit configurations, complex reflection groups, and arithmetic statistics; preprint, 2016.

\bibitem[Ch]{church} Thomas Church, $G$-modules and stability in homology; in preparation.

\bibitem[CEF]{fimodules} Thomas Church, Jordan Ellenberg, Benson Farb, FI-modules and stability for representations of symmetric groups, {\it Duke Math. J.} {\bf 164} (2015), no.~9, 1833--1910, \arxiv{1204.4533v4}.

\bibitem[CEFN]{fi-noeth} Thomas Church, Jordan S. Ellenberg, Benson Farb, Rohit Nagpal, FI-modules over Noetherian rings, {\it Geom. Top.} {\bf 18} (2014), no.~5, 2951--2984, \arxiv{1210.1854v2}.

\bibitem[De]{deligne} P.~Deligne, Cat\'egories tannakiennes, {\it The Grothendieck Festschrift}, Vol.~II, 111--195, Progr. Math., 87, Birkh\"auser Boston, Boston, MA, 1990.

\bibitem[Hal]{hall} P.~Hall, Finiteness conditions for soluble groups, {\it Proc. London Math. Soc. (3)} {\bf 4} (1954), 419--436.

\bibitem[HW]{hatcher-wahl} Allen Hatcher, Nathalie Wahl, Stabilization for mapping class groups of 3-manifolds, {\it Duke Math. J.} {\bf 155} (2010), no.~2, 205--269, \arxiv{0709.2173v4}.

\bibitem[HU]{hopcroftullman} John E.~Hopcroft, Jeffrey D.~Ullman, {\it Introduction to Automata Theory, Languages, and Computation}, Addison-Wesley Series in Computer Science, Addison-Wesley Publishing Co., Reading, Mass., 1979.

\bibitem[Iv]{ivanov} S.~V. Ivanov, Group rings of Noetherian groups, {\it Mat. Zametki} {\bf 46} (1989), no.~6, 61--66; translation in {\it Math. Notes} {\bf 46} (1989), no. 5-6, 929--933 (1990).

\bibitem[Ke]{kerber} Adalbert Kerber, {\it Representations of Permutation Groups. I}, Lecture Notes in Math. {\bf 240}, Springer, Berlin, 1971.

\bibitem[Ko]{kowalski} Emmanuel Kowalski, {\it An introduction to the representation theory of groups}, Graduate Studies in Mathematics {\bf 155}, American Mathematical Society, Providence, RI, 2014.

\bibitem[KM]{kupers-miller} Alexander Kupers, Jeremy Miller, Representation stability for homotopy groups of configuration spaces, {\it J. Reine Angew. Math.}, to appear, \arxiv{1410.2328v2}.

\bibitem[Ma]{macdonald-wr} I.~G. Macdonald, Polynomial functors and wreath products, {\it J. Pure Appl. Algebra} {\bf 18} (1980), no.~2, 173--204. 

\bibitem[NW]{nash-williams} C. St. J. A. Nash-Williams, On well-quasi-ordering finite trees, {\it Proc. Cambridge Philos. Soc.} {\bf 59} (1963), 833--835.

\bibitem[Pa]{palmer} Martin Palmer, Twisted homological stability for configuration spaces, \arxiv{1308.4397v2}.

\bibitem[PS]{putman-sam} Andrew Putman, Steven~V Sam, Representation stability and finite linear groups, \arxiv{1408.3694v2}.

\bibitem[Ra]{ramos} Eric Ramos, On the degree-wise coherence of $\cF\cI_G$-modules, \arxiv{1606.04514v1}.

\bibitem[SS1]{symc1} Steven~V Sam, Andrew Snowden, GL-equivariant modules over polynomial rings in infinitely many variables, {\it Trans. Amer. Math. Soc.} {\bf 368} (2016), 1097--1158, \arxiv{1206.2233v3}.

\bibitem[SS2]{expos} Steven~V Sam, Andrew Snowden, Introduction to twisted commutative algebras, \arxiv{1209.5122v1}.

\bibitem[SS3]{catgb} Steven~V Sam, Andrew Snowden, Gr\"obner methods for representations of combinatorial categories, {\it J. Amer. Math. Soc.}, to appear, \arxiv{1409.1670v3}.

\bibitem[Se]{serre} Jean-Pierre Serre, {\it Linear Representations of Finite Groups}, Graduate Texts in Mathematics {\bf 42}, Springer-Verlag, New York-Heidelberg, 1977.

\bibitem[Sn]{delta-mod} Andrew Snowden, Syzygies of Segre embeddings and $\Delta$-modules, {\it Duke Math.\ J.} {\bf 162} (2013), no.~2, 225--277, \arxiv{1006.5248v4}.

\bibitem[WR]{wahl} Nathalie Wahl, Oscar Randal-Williams, Homological stability for automorphism groups, \arxiv{1409.3541v3}.

\bibitem[W1]{wilson-repstab} Jennifer C.~H.~Wilson, Representation stability for the cohomology of the pure string motion groups, {\it Algebr. Geom. Topol.} {\bf 12} (2012), 909--931, \arxiv{1108.1255v3}.

\bibitem[W2]{wilson} Jennifer C.~H.~Wilson, ${\rm FI}_{\cW}$-modules and stability criteria for representations of classical Weyl groups, {\it J. Algebra} {\bf 420} (2014), 269--332, \arxiv{1309.3817v2}.

\end{thebibliography}
\end{document}